\documentclass[a4paper, 11pt]{amsart}   
\usepackage{mathptmx, amssymb,amscd,latexsym, eulervm}   
\usepackage{amsmath}
\usepackage{amsthm}
\usepackage{mathdots}
\usepackage{hyperref} 
\hypersetup{pagebackref,colorlinks=true,linkcolor=blue,urlcolor=blue,breaklinks=true}
\usepackage{color}     
\usepackage[onehalfspacing]{setspace}
\usepackage{tabularx}
\usepackage{amsfonts}
\usepackage{paralist}
\usepackage{aliascnt}
\usepackage[initials, lite]{amsrefs}
\usepackage{amscd}
\usepackage{blkarray}
\usepackage{mathbbol}
\usepackage{setspace}
\usepackage[inner=2.4cm,outer=2.4cm, bottom=3.2cm]{geometry}
\usepackage{tikz, tikz-cd}
\usepackage{calligra,mathrsfs}

\usepackage{tikz}
\usetikzlibrary{matrix}
\usetikzlibrary{arrows,calc}
\allowdisplaybreaks

\BibSpec{collection.article}{%
	+{}  {\PrintAuthors}                {author}
	+{,} { \textit}                     {title}
	+{.} { }                            {part}
	+{:} { \textit}                     {subtitle}
	+{,} { \PrintContributions}         {contribution}
	+{,} { \PrintConference}            {conference}
	+{}  {\PrintBook}                   {book}
	+{,} { }                            {booktitle}
	+{,} { }                            {series}
	+{, vol.} { }                            {volume}
	+{,} { }                            {publisher}
	+{,} { \PrintDateB}                 {date}
	+{,} { pp.~}                        {pages}
	+{,} { }                            {status}
	+{,} { \PrintDOI}                   {doi}
	+{,} { available at \eprint}        {eprint}
	+{}  { \parenthesize}               {language}
	+{}  { \PrintTranslation}           {translation}
	+{;} { \PrintReprint}               {reprint}
	+{.} { }                            {note}
	+{.} {}                             {transition}
	+{}  {\SentenceSpace \PrintReviews} {review}
}
\AtBeginDocument{%
	\def\MR#1{}
}

\makeatletter
\@namedef{subjclassname@2020}{%
	\textup{2020} Mathematics Subject Classification}
\makeatother

\newcommand{\Bl}{{\mathcal{B}\ell}}
\newcommand{\RR}{\mathbb{R}}
\newcommand{\kk}{\mathbb{k}}

\newcommand{\sP}{\mathscr{P}}
\newcommand{\sR}{\mathscr{R}}

\newcommand{\UU}{\mathbb{U}}

\newcommand{\NN}{\normalfont\mathbb{N}}
\newcommand{\ZZ}{\mathbb{Z}}

\newcommand{\PP}{{\normalfont\mathbb{P}}}

\newcommand{\mm}{{\normalfont\mathfrak{m}}}

\newcommand{\pp}{{\normalfont\mathfrak{p}}}

\newcommand{\bn}{{\normalfont\mathbf{n}}}

\newcommand{\ttt}{{\normalfont\mathbf{t}}}

\newcommand{\rank}{\normalfont\text{rank}}

\newcommand{\Sym}{\normalfont\text{Sym}}
\newcommand{\Rees}{\mathscr{R}}

\newcommand{\ee}{{\normalfont\mathbf{e}}}

\newcommand{\EE}{\mathcal{E}}

\newcommand{\OO}{\mathcal{O}}

\newcommand{\sX}{\mathscr{X}}
\newcommand{\sY}{\mathscr{Y}}
\newcommand{\sK}{\mathscr{K}}

\newcommand{\FF}{\mathcal{F}}

\newcommand{\HH}{\normalfont\text{H}}

\newcommand{\AAA}{\mathbb{A}}

\newcommand{\Proj}{\normalfont\text{Proj}}

\newcommand{\Spec}{\normalfont\text{Spec}}

\newcommand{\multProj}{\normalfont\text{MultiProj}}

\def\f0{\mathbf{0}}

\def\1{\mathbf{1}}

\newtheorem{theorem}{Theorem}[section]

\newtheorem{headthm}{Theorem}

\newaliascnt{headcor}{headthm}

\aliascntresetthe{headcor}

\newaliascnt{headconj}{headthm}

\aliascntresetthe{headconj}

\newaliascnt{corollary}{theorem}
\newtheorem{corollary}[corollary]{Corollary}
\aliascntresetthe{corollary}

\newaliascnt{claim}{theorem}

\aliascntresetthe{claim}

\newaliascnt{lemma}{theorem}
\newtheorem{lemma}[lemma]{Lemma}
\aliascntresetthe{lemma}

\newaliascnt{conjecture}{theorem}

\aliascntresetthe{conjecture}

\newaliascnt{proposition}{theorem}
\newtheorem{proposition}[proposition]{Proposition}
\aliascntresetthe{proposition}

\theoremstyle{definition}
\newaliascnt{definition}{theorem}
\newtheorem{definition}[definition]{Definition}
\aliascntresetthe{definition}

\newaliascnt{notation}{theorem}
\newtheorem{notation}[notation]{Notation}
\aliascntresetthe{notation}

\newaliascnt{example}{theorem}
\newtheorem{example}[example]{Example}
\aliascntresetthe{example}

\newaliascnt{examples}{theorem}

\aliascntresetthe{examples}

\newaliascnt{remark}{theorem}
\newtheorem{remark}[remark]{Remark}
\aliascntresetthe{remark}

\newaliascnt{question}{theorem}

\aliascntresetthe{question}

\newaliascnt{questions}{theorem}

\aliascntresetthe{questions}

\newaliascnt{problem}{theorem}

\aliascntresetthe{problem}

\newaliascnt{construction}{theorem}

\aliascntresetthe{construction}

\newaliascnt{setup}{theorem}
\newtheorem{setup}[setup]{Setup}
\aliascntresetthe{setup}

\newaliascnt{setupdef}{theorem}

\aliascntresetthe{setupdef}

\newaliascnt{algorithm}{theorem}

\aliascntresetthe{algorithm}

\newaliascnt{observation}{theorem}

\aliascntresetthe{observation}

\newaliascnt{defprop}{theorem}

\aliascntresetthe{defprop}

\newtheorem{chunk}[definition]{}

\def\equationautorefname~#1\null{(#1)\null}
\def\sectionautorefname~#1\null{Section #1\null}
\def\subsectionautorefname~#1\null{\S #1\null}

\def\surjects{\twoheadrightarrow}

\DeclareFontFamily{OT1}{pzc}{}
\DeclareFontShape{OT1}{pzc}{m}{it}{<-> s * [1.100] pzcmi7t}{}
\DeclareMathAlphabet{\mathchanc}{OT1}{pzc}{m}{it}

\DeclareMathOperator{\fSpec}{\mathchanc{Spec}}
\DeclareMathOperator{\fProj}{\mathchanc{Proj}}
\DeclareMathOperator{\fMProj}{\mathchanc{MultiProj}}

\title{Mixed Segre zeta functions and their log-concavity}

\author{Yairon Cid-Ruiz}
\address{Department of Mathematics, North Carolina State University, Raleigh, NC, USA}
\email{ycidrui@ncsu.edu}

\date{\today}

\begin{document}

\begin{abstract}
	We introduce and study the \emph{mixed Segre zeta function} of a sequence of homogeneous ideals in a polynomial ring.
	This function is a power series encoding information about the mixed Segre classes obtained by extending the ideals to projective spaces of arbitrarily large dimension.
	Our work generalizes and unifies results by Kleiman and Thorup on mixed Segre classes and by Aluffi on Segre zeta functions.
	We prove that this power series is rational, with poles corresponding to the degrees of the generators of the ideals.
	We also show that the mixed Segre zeta function only depends on the integral closure of the ideals. 
	Finally, we prove that the homogenization of the numerator of a modification of the mixed Segre zeta function is denormalized Lorentzian in the sense of Br\"and\'en and Huh.
\end{abstract}

	\maketitle
	
\section{Introduction}

 Let $\kk$ be an arbitrary field, $R = \kk[x_0,\ldots,x_n]$ be a polynomial ring,  and $I_1, \ldots, I_m \subset R$ be homogeneous ideals in $R$.
Denote by $Z_1,\ldots,Z_m \subset \PP_\kk^n$  the closed subschemes of $\PP_\kk^n$ defined by the ideals $I_1,\ldots,I_m \subset R$ and by $Z \subset \PP_\kk^n$ the closed subscheme of $\PP_\kk^n$ defined by the product ideal $I_1\cdots I_m \subset R$.
In their foundational study of \emph{mixed Buchsbaum-Rim multiplicities}, Kleiman and Thorup \cite{KLEIMAN_THORUP_MIXED} introduced the notion of \emph{mixed Segre classes}.
These classes yield a mixed generalization of the classical \emph{Segre classes}, which play a key role in Fulton--MacPherson’s intersection theory.
For each $i_1,\ldots,i_m \in \NN$, we have the \emph{mixed Segre class} 
$$
s^{i_1,\ldots,i_m}(Z_1, \ldots,Z_m; \PP_\kk^n) \;\in\; A_{n-i_1-\cdots-i_m}(Z)
$$ 
of $Z_1,\ldots,Z_m$ in $\PP_\kk^n$ of type $(i_1,\ldots,i_m)$.
For more details, see \autoref{def_KT_transform}.

\smallskip

Motivated by the \emph{Segre zeta function} of Aluffi \cite{aluffi2017segre}, we introduce the \emph{mixed Segre zeta function}.
The mixed Segre zeta function of the ideals $I_1,\ldots,I_m$ is a formal power series with integer coefficients 
$$
\zeta_{I_1, \ldots, I_m}(t_1,\ldots,t_m) \;=\; \sum_{i_1,\ldots,i_m \in \NN} \, a_{i_1,\ldots,i_m} \, t_1^{i_1}\cdots t_m^{i_m} \;\in\; \ZZ[\![t_1,\ldots,t_m]\!].
$$
This power series is characterized by the fact that, for all $N \ge n$ and $i_1+\cdots+i_m \le N$, the class 
$$
a_{i_1,\ldots,i_m} H^{i_1+\cdots + i_m} \smallfrown \left[\PP_\kk^N\right] \;\in\; A_{N-i_1-\cdots-i_m}\left(\PP_\kk^N\right)
$$
equals the push-forward to $\PP_\kk^N$ of the mixed Segre class $s^{i_1,\ldots,i_m}(Z_1^N,\ldots, Z_m^N; \PP_\kk^N)$; 
here $H$ denotes the class of a hyperplane in $\PP_\kk^N$ and $Z_i^N \subset \PP_\kk^N$ is the closed subscheme of $\PP_\kk^N$ defined by the extension of the ideal $I_i \subset R$  to the polynomial ring $\kk[x_0,\ldots,x_n,x_{n+1},\ldots,x_N] \supseteq R = \kk[x_0,\ldots,x_n]$.
Therefore, the mixed Segre zeta function $\zeta_{I_1, \ldots, I_m}(t_1,\ldots,t_m)$ can be viewed as the generating function that encodes the push-forwards of the corresponding mixed Segre classes obtained after extending the ideals $I_1,\ldots,I_m$ to arbitrarily large polynomial rings. 
Since mixed Segre classes can be seen as a refinement of Segre classes, mixed Segre zeta functions can likewise be regarded as a refinement of Segre zeta functions.
It is not difficult to show the existence of the mixed Segre zeta function (see \autoref{lem_existence}).
Our first main result shows that the mixed Segre zeta function is a \emph{rational} power series.

\begin{headthm}[{\autoref{thm_rationality}}]
	\label{thmA}
	Let $I_1, \ldots, I_m \subset R$ be homogeneous ideals in $R = \kk[x_0,\ldots,x_n]$.
	For each $1 \le i \le m$, let $f_{i, 0}, f_{i, 1}, \dotsc, f_{i, r_i} \in R$ be homogeneous generators of $I_i$, with the degree of $f_{i, j}$ equal to $d_{i,j} = \deg\left(f_{i,j}\right)$.
	Then we can write 
	$$
	\zeta_{I_1, \ldots, I_m}(t_1,\ldots,t_m) \;\;=\;\; \frac{P(t_1,\ldots,t_m)}{\;\prod_{1 \le i \le m} \left(1+d_{i,0}t_i\right)\left(1+d_{i,1}t_i\right)\cdots\left(1+d_{i,r_i}t_i\right)\;}
	$$
	where $P(t_1,\ldots,t_m) \in \NN[t_1,\ldots,t_m]$ is a polynomial with nonnegative integer coefficients.
\end{headthm}

Our proof of \autoref{thmA} is inspired by Aluffi's proof of the rationality of Segre zeta functions \cite{aluffi2017segre}.
The main technical tool is the analysis of how mixed Segre classes pull-back under certain rational maps (see \autoref{thm_pullback_Seg}).
This analysis, in turn, relies heavily on a blow-up formula for mixed Segre classes established in \autoref{thm_blow_up_form}.

\smallskip

We show that mixed Segre zeta functions have several interesting properties.  
Due to \autoref{lem_truncation}, we obtain that the mixed Segre zeta function $\zeta_{I_1,\ldots,I_m}(t_1,\ldots,t_m)$ is determined by the degrees of the generators of the ideals $I_1,\ldots,I_m$ and the mixed Segre classes $s^{i_1,\ldots,i_m}(Z_1^N,\ldots, Z_m^N; \PP_\kk^N)$ for $N$ high enough (explicitly, for $N \ge m + \sum_{i=1}^{m}r_i$).
It turns out that in general not all the numbers $-1/\deg(f_{i,j})$ appear as poles of $\zeta_{I_1,\ldots,I_m}(t_1,\ldots,t_m)$ (even if each set $f_{i, 0}, f_{i, 1}, \dotsc, f_{i, r_i}$ gives a minimal set of generators of each ideal $I_i$).
Indeed, by \autoref{prop_integral_dep}, mixed Segre zeta functions only depend on the integral closure of the ideals.
More precisely, we have the equality 
$$
\zeta_{I_1, \ldots, I_m}(t_1,\ldots,t_m) \;=\;  \zeta_{\overline{I_1}, \ldots, \overline{I_m}}(t_1,\ldots,t_m),
$$
where $\overline{I_i} \subset R$ denotes the integral closure of $I_i$.
In \autoref{prop_mixed_formula}, we provide a mixed formula that expresses the Segre zeta function $\zeta_{I_1 \cdots I_m}(t)$ of the product ideal $I_1\cdots I_m \subset R$ in terms of the mixed Segre zeta function $\zeta_{I_1, \ldots, I_m}(t_1,\ldots,t_m)$.
This mixed formula leads to a number of explicit formulas for the Segre zeta function of certain ideals (see \autoref{examp_binom_conv}, \autoref{exmp_r_1_s_2} and \autoref{examp_explicit}).

\smallskip

The second part of this paper investigates the log-concave behavior exhibited by mixed Segre zeta functions. 
Our main result in this direction is stated in the following theorem. 

\begin{headthm}[{\autoref{thm_Lorentzian}}]
	\label{thmB}
	Assume that $\kk$ is an algebraically closed field.  
	Let $I_1,\ldots,I_m \subset R$ be homogeneous ideals in $R = \kk[x_0,\ldots,x_n]$.
	Let $f_{i, 0}, f_{i, 1}, \dotsc, f_{i, r_i} \in R$ be homogeneous generators of $I_i$ with $d_{i,j} = \deg\left(f_{i,j}\right)$.
	Define the polynomial $Q(t_1,\ldots,t_m)$ by the identity
	$$
	1 - \zeta_{I_1, \ldots, I_m}(t_1,\ldots,t_m) \;=\; \frac{Q(t_1,\ldots,t_m)}{\;\prod_{1 \le i \le m} \left(1+d_{i,0}t_i\right)\left(1+d_{i,1}t_i\right)\cdots\left(1+d_{i,r_i}t_i\right)\;}.
	$$
	Then the homogenization of $Q(t_1,\ldots,t_m)$ is a denormalized Lorentzian polynomial.
\end{headthm}
	
This result yields a systematic method to produce Lorentzian polynomials from arbitrary homogeneous ideals in a polynomial ring. 	
It should be mentioned that another paper of Aluffi \cite{aluffi2024lorentzian} shows that the homogenization of the numerator of the power series $1 - \zeta_{I}(t_1,\ldots,t_m)$ associated to a  multihomogeneous ideal $I$ is a \emph{covolume polynomial} (a dual notion to that of a \emph{volume polynomial}; see \cite[\S 4.2]{BrandenHuh});
in this latter work, $I$ is a multihomogeneous ideal defining a closed subscheme of a product of projective spaces $\PP_\kk^{\ell_1} \times_\kk \cdots \times_\kk \PP_\kk^{\ell_m}$ and $\zeta_{I}(t_1,\ldots,t_m)$ is the corresponding Segre zeta function associated to $I$.
In \autoref{thmB}, we take a different approach: instead of considering a single multihomogeneous ideal $I$ and its Segre zeta function, we study a sequence of homogeneous ideals $I_1,\ldots,I_m$ and their mixed Segre zeta function. 
Owing to technical and favorable properties of mixed Segre classes, this change of perspective enables us to derive (the normalization of the homogenization of) the polynomial $Q(t_1,\ldots,t_m)$ in \autoref{thmB} from a volume polynomial (see the proofs of \autoref{thm_Lorentzian} and \autoref{prop_Lorentzian}).
Due to a result of Br\"and\'en and Huh \cite[Theorem 4.16]{BrandenHuh}, volume polynomials are always Lorentzian.
In contrast, covolume polynomials may fail to be Lorentzian, even after normalization (see, e.g., \cite[Example 2.7]{aluffi2024lorentzian}).
Nevertheless, Aluffi~\cite[Introduction]{aluffi2024lorentzian} conjectures that the homogenization of the numerator of $1 - \zeta_I(t_1, \ldots, t_m)$ is denormalized Lorentzian for any multihomogeneous ideal $I$ defining a closed subscheme of a product of projective spaces $\PP_\kk^{\ell_1} \times_\kk \cdots \times_\kk \PP_\kk^{\ell_m}$.

\medskip

\noindent
\textbf{Outline.}
The structure of the paper is as follows.
In \autoref{sect_mixed_Seg_blow_up}, we prove several results regarding mixed Segre classes, including a blow-up formula for them.
The behavior of the pull-back of mixed Segre classes under certain rational maps is studied in \autoref{sect_funct_Segre}.
We prove \autoref{thmA} in \autoref{sect_mixed_Seg_zeta}.
The proof of \autoref{thmB} is given in \autoref{sect_Lorentzian}.
Finally, we provide several examples and computations in \autoref{sect_examples}.

\section{Preliminaries}

In this section, we set up the notation that is used throughout the paper. 
We also present some preliminary results regarding multiprojective schemes and joint blow-ups.

Let $m \ge 1$ be a positive integer. 
If $\bn = (n_1,\ldots,n_m)$, $\bn' = (n_1',\ldots,n_m') \in \ZZ^m$ are two multi-indexes, we write $\bn \ge \bn'$ whenever $n_i \ge n_i'$ for all $1 \le i \le m$, and $\bn > \bn'$ whenever $n_j > n_j'$ for all $1 \le j \le m$.
For each $1 \le i \le m$, let $\ee_i \in \NN^m$ be the $i$-th elementary vector $\ee_i:=\left(0,\ldots,1,\ldots,0\right)$.
Let $\mathbf{0} \in \NN^m$ and $\mathbf{1} \in \NN^m$ be the vectors $\mathbf{0}:=(0,\ldots,0)$ and $\mathbf{1}:=(1,\ldots,1)$ of $m$ copies of $0$ and $1$, respectively.  

\begin{chunk}
	We say that an $\NN^m$-graded ring $B = \bigoplus_{\bn \in \NN^m} B_\bn$ is \emph{standard} if $B$ is generated as an algebra over $B_{\mathbf{0}}$ by finitely many elements of the form $b \in B_{\ee_i}$ for some $1 \le i \le m$.
	We always assume that all rings are
	Noetherian.
	The \emph{multigraded irrelevant ideal} is given by $B_{++} := \bigoplus_{\bn \ge \mathbf{1}} B_\bn$.
\end{chunk}

\begin{chunk}[{MultiProj of a standard multigraded ring; see~\cite[\S 1]{HYRY_MULTIGRAD}}]
	Let $B$ be a standard $\NN^m$-graded ring.
	The \emph{multiprojective scheme}  $\multProj(B)$ of $B$ is given by 
	$$
	X \;=\; \multProj(B) \;:=\; \big\{ \pp \in \Spec(B) \mid \pp \text{ is $\NN^m$-graded and } \pp \not\supseteq B_{++} \big\},
	$$
	and its scheme structure is obtained by using multi-homogeneous localizations. 
	We have a natural morphism $\pi : X \rightarrow \Spec(B_\mathbf{0})$. 
	A distinguished affine open cover of $X$ is given by 
	$$
	\OO_X\left(D(g_1\cdots g_m)\right) \;\cong\; \Spec\left(B_{(g_1\cdots g_m)}\right),
	$$ where $D(g_1\cdots g_m) := \left\lbrace \pp \in X \mid g_1\cdots g_m \notin \pp \right\rbrace$ and $B_{(g_1\cdots g_m)} := \left\lbrace f/(g_1\cdots g_m)^k\in B_{g_1\cdots g_m} \mid f \in B_{k \cdot \mathbf{1}}\right\rbrace$, for all $g_1 \in B_{\ee_1}, \ldots, g_m \in B_{\ee_m}$.
	A finitely generated $\ZZ^m$-graded $B$-module $M$ has an associated coherent sheaf $\widetilde{M}$ (the sheafification of $M$) on $X$.
\end{chunk}

\begin{chunk}[Relative MultiProj]
	\label{chunk_rel_mproj}
	Let $X$ be a Noetherian scheme. 
	Let $\mathcal{C} = \bigoplus_{\bn \in \NN^m}\mathcal{C}_\bn$  be a standard $\NN^m$-graded quasi-coherent $\OO_X$-algebra (i.e., $\mathcal{C}_{\mathbf{0}} = \OO_X$, each $\mathcal{C}_{\ee_i}$ is a coherent $\OO_X$-module, and $\mathcal{C}$ is generated over $\OO_X$ by $\mathcal{C}_{\ee_1}, \mathcal{C}_{\ee_2}, \ldots, \mathcal{C}_{\ee_m}$).
	\begin{enumerate}[\rm (a)]
		\item	For each open affine subscheme $U = \Spec(A) \subset X$, we  consider the standard $\NN^m$-graded $A$-algebra $\mathcal{C}(U) = \bigoplus_{\bn \in \NN^m}\mathcal{C}_\bn(U)$ and the  multiprojective scheme $\pi_U : \multProj(\mathcal{C}(U)) \rightarrow U$.
		By gluing all the morphisms $\pi_U$, we obtain the \emph{(relative) multiprojective scheme}
		$
		\pi : \sX = \fMProj_X\left(\mathcal{C}\right) \rightarrow X
		$
		of $\mathcal{C}$. 
		This multiprojective construction  mimics the (relative) $\fProj$ (see \cite[\S II.7]{HARTSHORNE}).
		A finitely generated $\ZZ^m$-graded $\mathcal{C}$-module $\mathcal{M}$ has an associated coherent sheaf $\widetilde{\mathcal{M}}$ (the sheafification of $\mathcal{M}$) on $\sX$.
		\item Consider the diagonal subalgebra $\Delta := \bigoplus_{k \ge 0} \mathcal{C}_{k \cdot \mathbf{1}} \subset \mathcal{C}$. 
		We have the natural isomorphism 
		$$
		\sX = \fMProj_X\left(\mathcal{C}\right) \;\xrightarrow{\;\;\cong\;\;} \fProj_X\left(\Delta\right).
		$$
		This isomorphism can be seen as a (relative) Segre embedding.
		\item Consider the standard $\NN^{m-1}$-graded algebra $\mathcal{B} = \bigoplus_{\bn \in \NN^{m-1}} \mathcal{C}_{\bn,0} \subset \mathcal{C}$.
		Let  $\sY = \fMProj_X(\mathcal{B})$. 
		Notice that $\mathcal{C} = \mathcal{B}\left[\mathcal{C}_{\ee_m}\right]$ is a standard $\NN$-graded algebra over $\mathcal{B}$.
		Indeed, each $\mathcal{C}_i := \mathcal{B} \cdot \mathcal{C}_{i \cdot \ee_m}= \bigoplus_{\bn \in \NN^{m-1}} \mathcal{C}_{\bn,i}$ is a finitely generated $\mathcal{B}$-module. 
		It follows that $\mathcal{S} := \bigoplus_{i \ge 0} \widetilde{\mathcal{C}_i}$ is a standard $\NN$-graded quasi-coherent $\OO_{\sY}$-algebra.
		We have the natural isomorphism 
		$
		\sX = \fMProj_X(\mathcal{C}) \cong \fProj_{\sY}\left(\mathcal{S}\right).
		$
		Thus the $\fMProj$ construction can be obtained by successively utilizing the $\fProj$ construction.
	\end{enumerate}
\end{chunk}

\begin{chunk}[Joint blow-up;  see \cite{KLEIMAN_THORUP_MIXED}]
	Let $X$ be a Noetherian scheme and $Z_1,\ldots,Z_m \subset X$ be closed subschemes.
	Let $Z \subset X$ be the closed subscheme defined by the product ideal sheaf $\mathcal{I}_{Z_1}\cdots \mathcal{I}_{Z_m} \subset \OO_X$.
	The multi-Rees algebra of $\mathcal{I}_{Z_1}, \ldots, \mathcal{I}_{Z_m} \subset \OO_X$ is given by
	$
	\sR = \sR_{\OO_{X}}\big(\mathcal{I}_{Z_1}, \ldots, \mathcal{I}_{Z_m}\big) \;:=\; \bigoplus_{(n_1,\ldots,n_m) \in \NN^m} \mathcal{I}_{Z_1}^{n_1} \cdots \mathcal{I}_{Z_m}^{n_m}.
	$
	\begin{enumerate}[\rm (a)]
		\item The \emph{joint blow-up} of $X$ along $Z_1,\ldots,Z_m$ is given by $\pi : \Bl_{Z_1,\ldots,Z_m}(X) := \fMProj_X\left(\sR\right) \rightarrow X$.
		\item We have the natural isomorphism $\Bl_{Z_1,\ldots,Z_m}(X) \xrightarrow{\;\;\cong\;\;} \Bl_Z(X)$.
		\item Let $\mathscr{Y} \subset \Bl_{Z_1,\ldots,Z_{m-1}}(X)$ be the closed subscheme defined by the ideal sheaf $\mathcal{I}_{Z_m} \OO_{\Bl_{Z_1,\ldots,Z_{m-1}}(X)}$.
		We have the natural isomorphism $\Bl_{Z_1,\ldots,Z_m}(X) \cong \Bl_{\mathscr{Y}}\left(\Bl_{Z_1,\ldots,Z_{m-1}}(X)\right)$.
		\item We have the inequality $\dim\left(\Bl_{Z_1,\ldots,Z_m}(X)\right) \le \dim(X)$, and an equality holds if each $Z_i$ is nowhere dense on $X$.
		More precisely, if $X_1,\ldots, X_s$ are the irreducible components of $X$, then 
		$$
		\dim\left(\Bl_{Z_1,\ldots,Z_m}(X)\right) \;=\; \max\big\lbrace \dim(X_i) \;\mid\; 1 \le i \le s \text{ and } \xi_i \in X\setminus Z \big\rbrace,
		$$
		where $\xi_i$ is the generic point of $X_i$.
		\item Assume $X = \Spec(R)$ is affine. Let $Z_i = V(I_i) \subset X$ for some ideal $I_i \subset R$. 
		Set $I = I_1\cdots I_m \subset R$ and $Z = V(I) \subset X$, and write $I = (f_1,\ldots,f_r)$. 
		A distinguished open cover of $\Bl_{Z_1,\ldots,Z_m}(X) \cong \Bl_Z(X)$ is given by
		$
		\Bl_Z(X) = \bigcup_{i=1}^r
		\Spec\left(R[I/f_i]\right),
		$
		where $R[I/f_i]$ denotes the $R$-subalgebra of $R_{f_i}$ generated by all $f/f_i$ with $f \in I$.
	\end{enumerate}
\end{chunk}
\begin{proof}
	Parts (a), (b) and (c) follow from \hyperref[chunk_rel_mproj]{\S 2.3}. 
	
	Part (d) follows, for instance, from \cite[\href{https://stacks.math.columbia.edu/tag/0BFM}{Lemma 0BFM}]{stacks-project},  \cite[Theorem 5.1.4]{huneke2006integral} or \cite[Proposition 12.14]{HIO} (see also \cite{KLEIMAN_THORUP_GEOM, KLEIMAN_THORUP_MIXED}).
	
	For part (e), see, for instance, \cite[\href{https://stacks.math.columbia.edu/tag/0804}{Lemma 0804}]{stacks-project} or \cite[Chapter 0, \S 3]{BENNET}.
\end{proof}

\begin{notation}
	We now recall some notation from \cite{KLEIMAN_THORUP_GEOM,KLEIMAN_THORUP_MIXED}.
	Let $X$ be a scheme, $W$ be a subscheme of $X$  and $\mathbf{S} \in Z_*(X)$ be a cycle on $X$. 
	We define the $W$-part $\mathbf{S}^W$ as follows: if $\mathbf{S}$ is the fundamental cycle of an integral closed subscheme $S \subset X$ of $X$, define $\mathbf{S}^W$ to be $\mathbf{S}$ if $W$ contains the generic point of $S$, and to be zero if not; then extend this definition by linearity.
\end{notation}

\begin{example}[{cf.~\cite[Remark 12.15]{HIO}}]
	Let $\kk$ be a field, $R = \kk[x,y,z,w]/\left(xzw, yzw\right)$ and $X = \Spec\left(R\right)$.
	Let $Z_1 = V(z) \subset X$ and $Z_2 = V(w) \subset X$. 
	Thus $\Bl_{Z_1, Z_2}(X) \cong  \Bl_Z(X)$, where $I = \left( zw\right) \subset R$ and $Z = V(I)$.
	Since $I$ is a principal ideal,  we get the isomorphism $\Bl_Z(X) = \Spec\left(R[I/zw]\right)$.
	A direct computation shows $R[I/zw] \cong R/(x,y)$.
	Therefore, we obtain that $\dim(\Bl_{Z_1, Z_2}(X)) = 2 <\dim(X) = 3$.
	Finally, notice that $\big[\Spec(R/(x,y))\big] = \big[\Spec(R/(0:_RI^\infty))\big] = \left[X\right]^{X \setminus Z} \in Z_*(X)$.
\end{example}

\section{Mixed Segre classes and blow-up formulas for them}	
\label{sect_mixed_Seg_blow_up}

In this section, we extend the mixed Segre classes introduced by Kleiman and Thorup \cite{KLEIMAN_THORUP_MIXED} to the case of more than two closed subschemes.
The respective extension of their joint blow-up construction was made in \cite{cid2023relative}.
The following setup is used throughout this section. 

\begin{setup}
	\label{setup_kleiman_thorup}
	Let $X$ be a scheme of finite type over a field $\kk$.
	Let $n = \dim(X)$ be the dimension of $X$.
	Let $m \ge 1$ be a positive integer and $Z_1,\ldots,Z_m \subset X$ be a sequence of closed subschemes of $X$.
	For each $1 \le i \le m$, let $\mathcal{J}_i \subset \OO_X$ be the ideal sheaf of $Z_i$.
	Let $\iota: Z \hookrightarrow X$ be the closed subscheme defined by the product ideal sheaf $\mathcal{J} = \mathcal{J}_1 \cdots \mathcal{J}_m \subset \OO_X$.
\end{setup}

Let $\widehat{X} := \fSpec_X\left(\OO_X[t]\right) = X \times_{\Spec(\kk)} \mathbb{A}_\kk^1$ be the relative affine line over $X$. 
We view $X$ as a closed subscheme embedded in $\widehat{X}$ as the principal divisor $\{t=0\}$.
Let $\widehat{\mathcal{B}} := \Bl_{Z_1,\ldots,Z_m }\big(\widehat{X}\big)$ be the joint blow-up of $\widehat{X}$ along the closed subschemes $Z_1,\ldots,Z_m$.
This scheme can be obtained by successively blowing-up $\widehat{X}$ along  $Z_1, \ldots,Z_m$.
It arises from the $\NN^m$-graded quasi-coherent $\OO_{\widehat{X}}$-algebra
$$
\sR_{\OO_{\widehat{X}}}\big((\mathcal{J}_1,t), \ldots, (\mathcal{J}_m,t)\big) \;:=\; \bigoplus_{(n_1,\ldots,n_m) \in \NN^m} (\mathcal{J}_1,t)^{n_1} \cdots (\mathcal{J}_m,t)^{n_m};
$$
i.e., the multi-Rees algebra of the ideal sheaves $(\mathcal{J}_1,t), \ldots, (\mathcal{J}_m,t) \subset \OO_{\widehat{X}}$ and 
$$
\widehat{\mathcal{B}} \;=\; \fMProj_{\widehat{X}}\left(\sR_{\OO_{\widehat{X}}}\big((\mathcal{J}_1,t), \ldots, (\mathcal{J}_m,t)\big)\right).
$$
Let $\widehat{b} : \widehat{\mathcal{B}} = \Bl_{Z_1,\ldots,Z_m }\big(\widehat{X}\big) \rightarrow \widehat{X}$ be the natural projection.
To simplify notation, for each $1 \le i \le m$, set $\mathcal{K}_i := \left(\mathcal{J}_i, t\right) \subset \OO_{\widehat{X}}$.
For each $1 \le i \le m$, let $\ee_i$ be the $i$-th elementary vector $\ee_i = (0,\ldots,1,\ldots,0) \in \NN^m$ and $\OO_{\widehat{\mathcal{B}}}(\ee_i)$ be the corresponding line bundle.

The main objects of interest in this section are the following remarkable constructions by Kleiman and Thorup \cite{KLEIMAN_THORUP_MIXED}.

\begin{definition}
	\label{def_KT_transform}
	For any closed subscheme $W \subset X$, we say that the \emph{Kleiman--Thorup transform} of $W$ is given by 
	$$
	P_W^{Z_1,\ldots,Z_m}(X) \;:=\; \widehat{\mathcal{B}} \times_{\widehat{X}} W,
	$$
	the restriction of $\widehat{\mathcal{B}}$ to the preimage of $W$.
	In the special cases $W = Z$ and $W = X$, we set 
	$$
	P_Z \;:= \; P_Z^{Z_1,\ldots,Z_m}(X) \qquad \text{ and } \qquad P_X \;:= \; P_X^{Z_1,\ldots,Z_m}(X).
	$$
	Let $p :  P_Z \rightarrow Z$ and $\pi : P_X \rightarrow X$  be the natural projections. 
	For each $i_1,\ldots,i_m \in \NN$, we say that the  \emph{mixed Segre class} of $Z_1,\ldots,Z_m$ in $X$ of type $(i_1,\ldots,i_m)$ is given by
	$$
	s^{i_1,\ldots,i_m}\left(Z_1,\ldots,Z_m; X\right) \;:=\; p_*\left(c_1\left(\OO_{P_Z}(\ee_1)\right)^{i_1}c_1\left(\OO_{P_Z}(\ee_2)\right)^{i_2}\cdots c_1\left(\OO_{P_Z}(\ee_m)\right)^{i_m} \smallfrown \left[P_Z\right]\right) \;\in\; A_*(Z).
	$$
	The \emph{total mixed Segre class} of $Z_1,\ldots,Z_m$ in $X$ is given by 
	\begin{align*}
		s\left(Z_1,\ldots,Z_m;X\right)  \;&:=\; \sum_{i_1,\ldots,i_m\in \NN} s^{i_1,\ldots,i_m}\left(Z_1,\ldots,Z_m; X\right)\, t_1^{i_1}\cdots t_m^{i_m} \\
		 \;&=\; p_*\left( \frac{1}{\big(1-c_1\left(\OO_{P_Z}(\ee_1)\right)t_1\big)\,\cdots\, \big(1-c_1\left(\OO_{P_Z}(\ee_m)\right)t_m\big)}\, \smallfrown\, [P_Z]\right).
	\end{align*}
	Of course, if $m= 1$, then $s(Z_1; X)$ recovers the usual Segre class of $Z_1$ in $X$ (see \cite[Chapter 4]{FULTON_INTERSECTION_THEORY}).
	Moreover, we can see $s\left(Z_1,\ldots,Z_m;X\right)$ as a formal polynomial in the variables $t_1,\ldots,t_m$ and with coefficients in $A_*(Z)$.
	By an abuse of notation, we write
	$$
	s\left(Z_1,\ldots,Z_m;X\right) \;\in\; A_*(Z)[\ttt] \;=\; A_*(Z)[t_1,\ldots,t_m].
	$$
	We regard $A_*(Z)[t_1,\ldots,t_m]$ as a group with the operation given by the addition of the formal polynomials in $A_*(Z)[t_1,\ldots,t_m]$.
\end{definition}

\begin{notation}
We also consider the joint blow-up 
$$
\mathcal{B} \;:=\; \Bl_{Z_1,\ldots,Z_m }\left(X\right) \;=\; \fMProj_{X}\left(\sR_{\OO_{X}}\big(\mathcal{J}_1, \ldots, \mathcal{J}_m\big)\right)
$$
of $X$ along $Z_1,\ldots,Z_m$.
Let $b : \mathcal{B} = \Bl_{Z_1,\ldots,Z_m }\big(X\big)\rightarrow X$ be the natural projection.
Let $E_i := V\left(\mathcal{J}_i \OO_\mathcal{B}\right)$ be the pull-back to $\mathcal{B}$ of the exceptional divisor of the blow-up $\Bl_{Z_i}(X)$ of $X$ along $Z_i$.
Notice that $\OO_\mathcal{B}(-E_i) = \OO_\mathcal{B}(\ee_i)$ for each $1 \le i \le m$. 
We have a natural isomorphism 
$$
\Bl_{Z_1,\ldots,Z_m}\left(X\right) \;\xrightarrow{\;\cong\;} \Bl_Z\left(X\right)
$$
identifying $\Bl_{Z_1,\ldots,Z_m}(X)$ with the blow-up $\Bl_Z(X)$ of $X$ along $Z$.
Furthermore,  $E :=E_1+ \cdots + E_m$ coincides with the exceptional divisor of $\Bl_Z(X)$.
Let $\eta : E \rightarrow Z$ be the natural projection.
\end{notation}

\begin{remark}
	\label{rem_dim_P_Z}
	We can check that $\dim\left(P_X\right) = n$,  $\dim\left(P_Z\right) = n$ and $\dim\left(\Bl_{Z_1,\ldots,Z_m }\big(X\big)\right) \le n$.
	We always have $\dim\left(\Bl_{Z_1,\ldots,Z_m }\big(X\big)\right) \le n$ and an equality holds if each $Z_i$ is nowhere dense on $X$.
	Since each $Z_i$ is nowhere dense on $\widehat{X}$, we obtain that the dimension of $\widehat{\mathcal{B}}=\Bl_{Z_1,\ldots,Z_m }\big(\widehat{X}\big)$  is equal to $n+1$.
	It then follows that $\dim(P_X) = n$ because $P_X=V(t\OO_{\widehat{\mathcal{B}}})$ is a divisor on $\widehat{\mathcal{B}}$.
	Beware: notice that $P_Z$ is not necessarily a (Cartier) divisor on $\widehat{\mathcal{B}}$.
	However,  $V\left(\mathcal{K}_1\cdots\mathcal{K}_m\OO_{\widehat{\mathcal{B}}}\right)$ is a divisor on $\widehat{\mathcal{B}}$ and so the inclusions $ \left(\mathcal{J}_1\cdots\mathcal{J}_m,\, t^m\right) \subset \mathcal{K}_1\cdots\mathcal{K}_m \subset \left(\mathcal{J}_1\cdots\mathcal{J}_m,\, t\right)$ yield $\dim\left(P_Z\right) = \dim\left(V\left(\mathcal{K}_1\cdots\mathcal{K}_m\OO_{\widehat{\mathcal{B}}}\right)\right) = n$.
	Moreover, if $X$ is equidimensional, then $P_X$ and $P_Z$ are also equidimensional of dimension $n$.
\end{remark}

The next example illustrates the above constructions in a simple case.

\begin{example}
	Let $X = \AAA_\kk^2 = \Spec(\kk[x,y])$ and $\widehat{X} = \AAA_\kk^3 = \Spec(\kk[x,y,t])$.
	Let $Z_1 = V(x) \subset X$, $Z_2 = V(y) \subset X$ and $Z = V(xy) \subset X$.
	Let $S = \kk[x,y,t]$ and $K = (x,t)(y,t) = (xy, xt, yt, t^2) \subset S$.
	Then 
	\begin{align*}
		\widehat{\mathcal{B}} = \Bl_{Z_1, Z_2}(\widehat{X}) &\;=\; \Proj\left(\Rees_S(K)\right) \\
		&\;=\; \Spec\big(S\left[K/xy\right]\big) \cup \Spec\big(S\left[K/xt\big]\right) \cup \Spec\big(S\left[K/yt\big]\right) \cup \Spec\big(S\left[K/t^2\right]\big).
	\end{align*}
	Both $P_X = V(t\OO_{\widehat{\mathcal{B}}})$ and $V\left(K\OO_{\widehat{\mathcal{B}}}\right)$ are effective Cartier divisors on $\widehat{\mathcal{B}}$.
	Let $A = S\left[K/xy\right]$ and $U = \Spec\left(A\right)$.
	The isomorphism $A = \kk\left[x,y, t/x, t/y\right] \cong \kk\left[x,y,u,v\right]/\left(xu-yv\right)$ follows under the identifications $u \mapsto t/x$ and $v \mapsto t/y$.
	On the affine chart $U = \Spec(A)$, we get that $P_Z = V\left((xy, t)\OO_{\widehat{\mathcal{B}}}\right)$ is defined by the ideal $I = (xy, xu) = x(y,u) \subset A$.
	Consider the maximal ideal $\mm = (x,y,u,v) \subset A$.
	By Nakayama's lemma, the minimal number of generators of $I$ at $\mm$ is equal to $\dim_\kk\left(I/\mm I\right) = 2$. 
	This shows that $P_Z$ is not determined by one equation at the point $\mm \in U = \Spec(A) \subset \widehat{\mathcal{B}}$.
	Therefore, $P_Z = V\left((xy, t)\OO_{\widehat{\mathcal{B}}}\right)$ is not a Cartier divisor on $\widehat{\mathcal{B}}$. 	
\end{example}

The next proposition yields an important result for studying mixed Segre classes.

\begin{proposition}
	\label{prop_decomp_P_X}
	Assume that $X$ is equidimensional.
	Then we have the following equalities
	\begin{align*}
		\big[P_X\big] \;=\;  \big[P_Z\big] + \big[\mathcal{B}\big], \quad
		\big[P_Z\big] \;=\; \big[P_X\big]^{\pi^{-1}\left(Z\right)} \quad\text{and} \quad \big[\mathcal{B}\big] \;=\; \big[P_X\big]^{\pi^{-1}\left(X \setminus Z\right)}
	\end{align*}
	of cycles in $P_X$.
\end{proposition}
\begin{proof}
	By \autoref{rem_dim_P_Z}, both $P_X$ and $P_Z$ are equidimensional of dimension $n$.
	To simplify notation, let $\sR := \sR_{\OO_{X}}\big(\mathcal{J}_1, \ldots, \mathcal{J}_m\big)$ and $\widehat{\sR} := \sR_{\OO_{\widehat{X}}}\big(\mathcal{K}_1, \ldots, \mathcal{K}_m\big)$.
	We have the short exact sequence 
	$$
	0 \;\rightarrow \; \mathscr{Q} \;\rightarrow\; \widehat{\sR}/t\widehat{\sR} \;\rightarrow\; \widehat{\sR}/\left(\mathcal{K}_1\cdots\mathcal{K}_m,\, t\right)\widehat{\sR} \rightarrow 0
	$$
	where $\mathscr{Q}$ is the $\NN^m$-graded quasi-coherent $\OO_{\widehat{X}}$-module given by 
	$$
	\mathscr{Q} \;:=\;  \bigoplus_{(n_1,\ldots,n_m) \in \NN^m}\; \frac{\mathcal{K}_1^{n_1+1}\cdots\mathcal{K}_m^{n_m+1}}{\mathcal{K}_1^{n_1+1}\cdots\mathcal{K}_m^{n_m+1} \,\cap\, t\mathcal{K}_1^{n_1}\cdots\mathcal{K}_m^{n_m}}.
	$$
	There is a natural surjection 
	$$
	\mathscr{Q} \;\;\surjects\;\; \sR(1,\ldots,1) \;\cong\; \bigoplus_{(n_1,\ldots,n_m) \in \NN^m}\; \frac{\mathcal{K}_1^{n_1+1}\cdots\mathcal{K}_m^{n_m+1}}{\mathcal{K}_1^{n_1+1}\cdots\mathcal{K}_m^{n_m+1} \,\cap\, t\OO_{\widehat{X}}}.
	$$	
	We have $P_X = \fMProj_{X}\big(\widehat{\sR}/t\widehat{\sR}\big)$, $P_Z = \fMProj_{X}\big(\widehat{\sR}/\left(\mathcal{K}_1\cdots\mathcal{K}_m,\, t\right)\widehat{\sR} \big)$ and $\mathcal{B} = \fMProj_{X}\big(\sR\big)$.
	Let $\mathcal{F}$ be the coherent $\OO_{P_X}$-module associated with $\mathscr{Q}$.
	We claim that the equality $[\mathcal{F}] = [\mathcal{B}] \in Z_*(P_X)$ holds in $P_X$.
	This equality of cycles follows by showing that $\left[\mathscr{Q}\right]_{(n_1,\ldots,n_m)}=\left[\mathscr{R}\right]_{(n_1+1,\ldots,n_m+1)}$ for all $n_1\gg0,\ldots,n_m\gg0$.
	The desired equality $\left[\mathscr{Q}\right]_{(n_1,\ldots,n_m)}=\left[\mathscr{R}\right]_{(n_1+1,\ldots,n_m+1)}$ holds whenever we have the inclusion 
	$$
	\mathcal{K}_1^{n_1+1}\cdots \mathcal{K}_m^{n_m+1} \,\cap\, t\OO_{\widehat{X}} \;\subset\; t \mathcal{K}_1^{n_1}\cdots\mathcal{K}_m^{n_m}.
	$$
	This inclusion can be checked by expanding  $(\mathcal{J}_1,t)^{n_1+1} \cdots (\mathcal{J}_m,t)^{n_m+1}$ and $t(\mathcal{J}_1,t)^{n_1} \cdots (\mathcal{J}_m,t)^{n_m}$.
	It is clear that the support of $[P_Z]$ is contained in $\pi^{-1}(Z)$.
	Finally, the statements of the proposition follow because the generic points of the irreducible components of $\mathcal{B} = \Bl_{Z_1,\ldots,Z_m}(X)$ are all contained in $\pi^{-1}(X\setminus Z)$.
\end{proof}

It is convenient to consider the formal push-forward 
$$
\iota_*\left(s\left(Z_1,\ldots,Z_m;X\right) \right) \;:=\; \sum_{i_1,\ldots,i_m\in \NN} \iota_*\left(s^{i_1,\ldots,i_m}\left(Z_1,\ldots,Z_m; X\right)\right)\, t_1^{i_1}\cdots t_m^{i_m}
$$
of the total mixed Segre class $s\left(Z_1,\ldots,Z_m;X\right)$.
Hence we can regard $\iota_*\left(s\left(Z_1,\ldots,Z_m;X\right)\right)$ as a formal polynomial in the variables $t_1,\ldots,t_m$ and with coefficients in $A_*(X)$.
	
The main result of this section is the following blow-up formula for mixed Segre classes (cf.~\cite[\S 5]{KLEIMAN_THORUP_MIXED}, \cite[Corollary 4.2.2, Example 4.2.1]{FULTON_INTERSECTION_THEORY}, \cite[Proposition 3.1]{ALUFFI_COMPUTING}).	
To simplify notation, let $\underline{Z}=Z_1,\ldots,Z_m$.	
	
\begin{theorem}[Blow-up formula]
	\label{thm_blow_up_form}
	Assume \autoref{setup_kleiman_thorup} and that $X$ is equidimensional.
	Then the statements below hold:
	\begin{enumerate}[\rm (i)]
		\item Modulo rational equivalence on $X$, we have
		$$
		\iota_*\left(s\left(\underline{Z};X\right) \right) \;=\; \left[X\right] \;-\; b_*\left(\frac{1}{\big(1+E_1t_1\big)\cdots \big(1+E_mt_m\big)} \,\smallfrown\, \big[\Bl_{Z}(X)\big]\right) \;\in\; A_*(X)[\ttt].
		$$
		\item Modulo rational equivalence on $Z$, we have
		$$
		s\left(\underline{Z};X\right) \;=\; \left[X\right]^Z \;+\; \eta_*\left(\frac{
			\big(1+E_1t_1\big)\cdots \big(1+E_mt_m\big) - 1
			}{\big(1+E_1t_1\big)\cdots \big(1+E_mt_m\big)} \,\smallfrown\, \big[\Bl_{Z}(X)\big]\right) \;\in\; A_*(Z)[\ttt].
		$$
		In other words, we have
		$$
		s^{i_1,\ldots,i_m}\left(\underline{Z};X\right) \;=\; \begin{cases}
			\left[X\right]^{Z} & \text{ if }\; i_1=\cdots=i_m=0\\
			{(-1)}^{i_1+\cdots+i_m-1}\;\eta_*\left(E_1^{i_1}\cdots E_m^{i_m} \cdot \big[\Bl_{Z}(X)\big]\right) & \text{ otherwise.}
		\end{cases}
		$$
	\end{enumerate}
\end{theorem}	
\begin{proof}
	(i)	
	We consider the fibre square 
		\begin{equation*}
		\begin{tikzpicture}[baseline=(current  bounding  box.center)]
			\matrix (m) [matrix of math nodes,row sep=3em,column sep=4.5em,minimum width=2em, text height=1.5ex, text depth=0.25ex]
			{
				P_Z & P_X \\
				Z & X.\\
			};
			\path[-stealth]
			(m-1-1) edge node [above] {$j$} (m-1-2)
			(m-2-1) edge node [above] {$\iota$} (m-2-2)
			(m-1-1) edge node [left] {$p$} (m-2-1)
			(m-1-2) edge node [left] {$\pi$} (m-2-2)
			;		
		\end{tikzpicture}	
	\end{equation*}
	Then we can deduce the following equalities 
	\begin{alignat*}{3}
			\iota_*\left(s\left(\underline{Z};X\right) \right) & \;=\; \iota_*\left(p_*\left(\prod_{i=1}^m\frac{1}{\big(1-c_1\left(\OO_{P_Z}(\ee_i)\right)t_i\big)}\, \smallfrown\, [P_Z]\right)\right) & \quad\text{ (by definition)}\\
			& \;=\; \pi_*\left(j_*\left(\prod_{i=1}^m\frac{1}{\big(1-c_1\left(\OO_{P_Z}(\ee_i)\right)t_i\big)}\, \smallfrown\, [P_Z]\right)\right) & \quad\text{ (by functoriality of push-forward)}\\
			& \;=\; \pi_*\left(\prod_{i=1}^m\frac{1}{\big(1-c_1\left(\OO_{P_X}(\ee_i)\right)t_i\big)}\, \smallfrown\, [P_Z]\right) & \quad\text{ (by the projection formula)}\\
			& \;=\; \pi_*\left(\prod_{i=1}^m\frac{1}{\big(1-c_1\left(\OO_{P_X}(\ee_i)\right)t_i\big)}\, \smallfrown\, \big([P_X] - [\mathcal{B}]\big)\right) & \quad\text{ (by \autoref{prop_decomp_P_X}).}
	\end{alignat*}
	For the summand involving $[\mathcal{B}]$, the projection formula and the fact that $\OO_\mathcal{B}(-E_i) = \OO_\mathcal{B}(\ee_i)$ yield
$$
		 \pi_*\left(\prod_{i=1}^m\frac{1}{\big(1-c_1\left(\OO_{P_X}(\ee_i)\right)t_i\big)}\, \smallfrown\, [\mathcal{B}]\right) \;=\;  b_*\left(\prod_{i=1}^m\frac{1}{\big(1+E_it_i\big)}\, \smallfrown\, [\mathcal{B}]\right).
$$
It remains to show that the other summand involving $[P_X]$ is equal to $[X]$.
Since each $Z_i$ is nowhere dense on $\widehat{X}$, the map $\widehat{b} : \widehat{\mathcal{B}} \rightarrow \widehat{X}$ is birational and so we get $\widehat{b}_*\left(\big[\widehat{\mathcal{B}}\big]\right)=\widehat{X}$.
We view $X$ and $P_X$ as divisors on $\widehat{X}$ and $\widehat{\mathcal{B}}$, respectively. 
Then the projection formula gives 
\begin{equation}
	\label{eq_project_P_X}
	\pi_*\left([P_X]\right) \;=\; \pi_*\left(P_X \,\cdot\, \big[\widehat{\mathcal{B}}\big]\right) \;=\; X \,\cdot\, \widehat{b}_*\left(\big[\widehat{\mathcal{B}}\big]\right)\;=\; X \,\cdot\, \big[\widehat{X}\big] \;=\; [X] \;\in\; A_*(X).
\end{equation}
Consider the divisor $\widehat{E}_i = V(\mathcal{K}_i\OO_{\widehat{\mathcal{B}}})$ on $\widehat{\mathcal{B}}$ and recall that $\OO_{\widehat{\mathcal{B}}}(-\widehat{E}_i) = \OO_{\widehat{\mathcal{B}}}(\ee_i)$.
Since $|\widehat{E}_i| \subset |P_X|$,  we have the restriction $j_i : \widehat{E}_i \hookrightarrow P_X$.
The class $[P_X] = 0 \in A_*(\widehat{\mathcal{B}})$ is rationally equivalent to zero in $\widehat{\mathcal{B}}$ because $P_X=V(t\OO_{\widehat{\mathcal{B}}})$ is a principal divisor on $\widehat{\mathcal{B}}$.
It then follows that $\widehat{E}_i \,\cdot\, [P_X] = 0 \in A_{*}(|\widehat{E}_i|)$. 
By utilizing the Gysin map of $\widehat{E}_i$ in $P_X$ (see \cite[\S 2.6]{FULTON_INTERSECTION_THEORY}), we obtain 
\begin{equation}
	\label{eq_vanish_P_X}
	-c_1\left(\OO_{P_X}(\ee_i)\right) \,\smallfrown\, [P_X] \;=\; {j_i}_*\left(j_i^*\left([P_X]\right)\right) \;=\; {j_i}_*\left(\widehat{E}_i \,\cdot\, [P_X]\right) \;=\; 0 \;\in\; A_{*}(P_X). 
\end{equation}
Consequently, we get 
$$
\pi_*\left(\prod_{i=1}^m\frac{1}{\big(1-c_1\left(\OO_{P_X}(\ee_i)\right)t_i\big)}\, \smallfrown\, [P_X]\right) \;=\;  [X].
$$
This concludes the proof of this part of the theorem.

\smallskip

(ii) The proof of this part follows similarly to part (i).
First, we note that $
s^{0,\ldots,0}\left(\underline{Z}; X\right) \;=\; \left[X\right]^Z \;\in\; A_n(Z)=Z_n(Z). 
$
Indeed, by combining \autoref{prop_decomp_P_X} and \autoref{eq_project_P_X}, we obtain 
$$
s^{0,\ldots,0}\left(\underline{Z}; X\right)  \;=\; p_*\left(\big[P_Z\big]\right) \;=\; p_*\left(\big[P_X\big]^{\pi^{-1}(Z)}\right) \;=\; \pi_*\left(\big[P_X\big]\right)^Z \;=\; \big[X\big]^Z.
$$
For $i_1+\cdots+i_m>0$, by definition we have 
$$
s^{i_1,\ldots,i_m}\left(\underline{Z}; X\right)  \;=\; p_*\left({(-1)}^{i_1+\cdots+i_m}\,\widehat{E}_1^{i_1} \cdots \widehat{E}_m^{i_m} \,\cdot \, \left[P_Z\right]\right).
$$
Similarly to \autoref{eq_vanish_P_X}, we have 
${\widehat{E}}_i \cdot [P_X] = 0 \in A_*(P_Z)$.
Hence \autoref{prop_decomp_P_X} yields
$$
\widehat{E}_i \,\cdot\, \left[P_Z\right] \;=\, - \widehat{E}_i \,\cdot\, \left[\mathcal{B}\right] \;\in\; A_*(P_Z).
$$
Since $|\widehat{E}_i| \cap \left[\mathcal{B}\right] \subset |E|$, this implies that
$$
\widehat{E}_1^{i_1} \cdots \widehat{E}_m^{i_m} \,\cdot \, \left[P_Z\right] \;=\; -\widehat{E}_1^{i_1} \cdots \widehat{E}_m^{i_m} \,\cdot\, \left[\mathcal{B}\right] 
\;=\; -E_1^{i_1} \cdots E_m^{i_m} \,\cdot\, \left[\mathcal{B}\right]
\;\in\; A_*(E)
$$
for any $i_1 + \cdots + i_m > 0$.
From these observations the formula of part (ii) follows.
\end{proof}

\begin{remark}
	\label{rem_seg_fulton}
	If $X$ is equidimensional and each $Z_i$ is nowhere dense on $X$, then \autoref{thm_blow_up_form}(ii) gives
	\begin{align*}
	s\left(\underline{Z};X\right) &\;=\; \sum_{i_1+\cdots+i_m \ge 1} {(-1)}^{i_1+\cdots+i_m-1}\;\eta_*\left(E_1^{i_1}\cdots E_m^{i_m} \cdot \big[\mathcal{B}\big]\right)\, t_1^{i_1}\cdots t_m^{i_m}	\\
	&\;=\; \frac{\left[E_1\right]t_1}{\big(1+E_1t_1\big)\cdots \big(1+E_mt_m\big)} \,+\, \frac{\left[E_2\right]t_2}{\big(1+E_2t_2\big)\cdots \big(1+E_mt_m\big)} \,+\,\cdots \,+\, \frac{\left[E_m\right]t_m}{\big(1+E_mt_m\big)}.
	\end{align*}
	When $m=1$, this recovers a well-known formula for Segre classes (see \cite[Corollary 4.2.2]{FULTON_INTERSECTION_THEORY}).
\end{remark}

We now prove some further results for mixed Segre classes that follow from \autoref{thm_blow_up_form}(ii). 
We first have a mixed formula relating mixed Segre classes with the usual Segre classes.

\begin{corollary}[Mixed formula]
	\label{cor_mixed_formula}
	Assume that $X$ is equidimensional.
	Then, for all $r \ge 0$, we have 
	$$
	s^r\left(Z; X\right) \;=\; \sum_{i_1+\cdots+i_m = r}\frac{r!}{i_1!\cdots i_m!} \,s^{i_1,\ldots,i_m}\left(Z_1,\ldots,Z_m; X\right) \;\in\; A_{n-r}(Z).
	$$
\end{corollary}
\begin{proof}
	Due to \autoref{thm_blow_up_form}(ii), we have $s^0(Z;X) = [X]^Z = s^{0,\ldots,0}(Z_1,\ldots,Z_m;X)$. 
	So we may assume $r > 0$.
	By applying \autoref{thm_blow_up_form}(ii) to both $s(Z;X)$ and $s(Z_1,\ldots,Z_m;X)$, we obtain 
	$$
		s^r(Z;X) \;=\; {(-1)^{r-1}}\,\eta_*\big((E_1+\cdots+E_m)^r\cdot [\mathcal{B}]\big) \;=\;  \sum_{i_1+\cdots+i_m = r}\frac{r!}{i_1!\cdots i_m!} \, s^{i_1,\ldots,i_m}\left(Z_1,\ldots,Z_m; X\right).
	$$
	So the proof of the corollary is complete.
\end{proof}

Next we study the behavior of mixed Segre classes under proper push-forward and flat pull-back (cf.~\cite[Proposition 4.2]{FULTON_INTERSECTION_THEORY}).
In particular, we obtain the birational invariance of mixed Segre classes.

\begin{theorem}
	\label{thm_funct_mixed_Segre}
	Assume \autoref{setup_kleiman_thorup}.
	Let $f : Y \rightarrow X$ be a morphism of equidimensional schemes and $Z_1' = f^{-1}(Z_1), \ldots, Z_m' = f^{-1}(Z_m)$, $Z' = f^{-1}(Z)$ be the inverse image schemes.
	Assume that each $Z_i$ is nowhere dense on $X$ and each $Z_i'$ is nowhere dense on $Y$. 
	Denote by $g : Z' \rightarrow Z$ the induced morphism.
	\begin{enumerate}[\rm (i)]
		\item {\rm(Proper push-forward)} 
		If $f$ is proper, $X$ is irreducible, and $f$ maps each irreducible component of $Y$ onto $X$, then 
		$$
		g_*\left(s(Z_1',\ldots,Z_m'; Y)\right) \;=\; \deg(Y/X) \, s(Z_1,\ldots,Z_m; X),
		$$ 
		where $\deg(Y/X) = \sum_{i=1}^{r}m_i \deg(Y_i/X)$, $Y_1,\ldots,Y_r$ are the irreducible components of $Y$, and $m_i$ is the geometric multiplicity of $Y_i$.
		\item {\rm(Flat pull-back)} If $f$ is flat of relative dimension $k$, then 
		$$
		g^*\left(s(Z_1,\ldots,Z_m; X)\right) \;=\; s(Z_1',\ldots,Z_m'; Y).
		$$
	\end{enumerate}
\end{theorem}
\begin{proof}
	Let $\Bl_{Z'}(Y) \cong \Bl_{Z_1',\ldots,Z_m'}(Y)$ be the blow-up of $Y$ along $Z'$ (equivalently, the joint blow-up of $Y$ along $Z_1',\ldots,Z_m'$).
	Let $E_i'$ be the pull-back to $\Bl_{Z'}(Y)$ of the exceptional divisor of the blow-up $\Bl_{Z_i'}(Y)$.
	Consider the divisor $E':=E_1'+\cdots+E_m'$ and the natural projection $\eta' : E' \rightarrow Z'$.
 	We have an induced morphism $F : \Bl_{Z'}(Y) \rightarrow \Bl_Z(X)$ such that $F^*E_i = E_i'$ as Cartier divisors.
 	Let $G : E' \rightarrow E$ be the induced morphism.
	Since each $Z_i$ is nowhere dense on $X$ and each $Z_i'$ is nowhere dense on $Y$, we may restrict to mixed Segre classes of type $(i_1,\ldots,i_m)$ with $i_1+\cdots+i_m > 0$.
	
	(i) In this case, $F$ and $G$ are proper and $F_*\left(\left[\Bl_{Z'}(Y)\right]\right) = d \left[\Bl_Z(X)\right]$ where $d = \deg(Y/X)$.
	By \autoref{thm_blow_up_form}(ii) and the projection formula (see \cite[Proposition 2.3]{FULTON_INTERSECTION_THEORY}), we obtain 
	\begin{align*}
	g_*\left(s^{i_1,\ldots,i_m}(Z_1',\ldots,Z_m'; Y)\right) &\;=\; g_*\left({(-1)}^{i_1+\cdots+i_m-1}\;\eta_*'\left((E_1')^{i_1}\cdots (E_m')^{i_m} \cdot \big[\Bl_{Z'}(Y)\big]\right)\right) \\
	&\;=\; {(-1)}^{i_1+\cdots+i_m-1}\; \eta_*\left(G_*\left((E_1')^{i_1}\cdots (E_m')^{i_m} \cdot \big[\Bl_{Z'}(Y)\big]\right)\right) \\
	&\;=\; d\,{(-1)}^{i_1+\cdots+i_m-1}\; \eta_*\left(E_1^{i_1}\cdots E_m^{i_m} \cdot \big[\Bl_{Z}(X)\big]\right) \\
	&\;=\; d\, s\left(Z_1,\ldots,Z_m; X\right). 
	\end{align*}
	
	(ii) Now $F$ and $G$ are both flat.
	From \autoref{thm_blow_up_form}(ii), \cite[Proposition 1.7]{FULTON_INTERSECTION_THEORY} and \cite[Proposition 2.3]{FULTON_INTERSECTION_THEORY}, we get
	\begin{align*}
	g^*\left(s^{i_1,\ldots,i_m}(Z_1,\ldots,Z_m; X)\right) &\;=\; g^*\left({(-1)}^{i_1+\cdots+i_m-1}\; \eta_*\left(E_1^{i_1}\cdots E_m^{i_m} \cdot \big[\Bl_{Z}(X)\big]\right)\right) \\
	&\;=\; {(-1)}^{i_1+\cdots+i_m-1}\;\eta_*'\left(G^*\left(E_1^{i_1}\cdots E_m^{i_m} \cdot \big[\Bl_{Z}(X)\big]\right)\right) \\
	&\;=\; {(-1)}^{i_1+\cdots+i_m-1}\;\eta_*'\left((E_1')^{i_1}\cdots (E_m')^{i_m} \cdot \big[\Bl_{Z'}(Y)\big]\right) \\
	&\;=\; s^{i_1,\ldots,i_m}(Z_1',\ldots,Z_m'; Y).
	\end{align*}
	This concludes the proof of the theorem.
\end{proof}

\subsection{Computing mixed Segre classes of subschemes of $\PP_\kk^n$}

In this subsection, we restrict to the case of a projective space.
Our goal is to provide an algorithm to compute mixed Segre classes, thus extending an algorithm of Aluffi \cite{ALUFFI_COMPUTING} to compute Segre classes (also, see \cite{HELMER_ALGO, EJP}).  
Now we utilize the following more specialized setting. 

\begin{setup}
	\label{setup_computing}
	Assume \autoref{setup_kleiman_thorup} with $X = \PP_\kk^n$.
	Let $R = \kk[x_0,\ldots,x_n]$ be the homogeneous coordinate ring of $\PP_\kk^n$.	
	For each $1 \le i \le m$, let $I_i = \left(f_{i,0},f_{i,1},\ldots,f_{i,r_i}\right) \subset R$ be a homogeneous ideal generated by $r_i+1$ forms of degree $d_i \ge 1$ such that $Z_i = V(I_i) \subset \PP_\kk^n$.\footnote{Assuming each $I_i$ is generated in a single degree is not restrictive. 
	For instance, if $J \subset R$ is homogeneous ideal, any truncation $J_{\ge d} = \bigoplus_{j \ge d}J_j$ defines the exact same closed subscheme in $\PP_\kk^n$.}
\end{setup}

Let $\Gamma := \Gamma_{I_1,\ldots,I_m} \subset \PP_\kk^n \times_\kk \PP_\kk^{r_1} \times_\kk \cdots \times_\kk \PP_\kk^{r_m}$  be the (closure) of the graph of the rational map
$$
\Phi_{I_1,\ldots,I_m} \,:\, \PP_\kk^n \;\dashrightarrow\; \PP_\kk^{r_1} \times_\kk \cdots \times_\kk \PP_\kk^{r_m}
$$
determined by the forms $f_{i,j}$.
Let $H$ be the class of a hyperplane in $\PP_\kk^n$ and $K_i$ be the class of a hyperplane in $\PP_\kk^{r_i}$.
Denote also by $H$ and $K_i$ their pullbacks to $\PP_\kk^n \times_\kk \PP_\kk^{r_1} \times_\kk \cdots \times_\kk \PP_\kk^{r_m}$.
As before, let  $b : \mathcal{B} = \Bl_{Z_1,\ldots,Z_m }\big(\PP_\kk^n\big)\rightarrow \PP_\kk^n$ be the joint blow-up and  $E_i$ be the pull-back to $\mathcal{B}$ of the exceptional divisor of the blow-up $\Bl_{Z_i}(\PP_\kk^n)$ of $X$ along $Z_i$.
We have a natural isomorphism $\Gamma \cong \mathcal{B}$ with the identification 
$$
\OO_\Gamma(\ee_i) \;\cong\; \OO_\mathcal{B}(\ee_i) \,\otimes\, \OO_{\PP_\kk^n}(d_i). 
$$
It then follows that $E_i = d_iH - K_i$.
The \emph{projective degrees} of the rational map $\Phi:=\Phi_{I_1,\ldots,I_m}$ are given by 
$$
d_{i_1,\ldots,i_m}\left(\Phi\right)  \;:=\; \int H^{n-i_1-\cdots-i_m} K_1^{i_1}\cdots K_m^{i_m} \cdot \left[\Gamma\right] 
$$
for each $i_1,\ldots,i_m \in \NN$.
The following corollary yields an explicit description of the mixed Segre class $\iota_*\left(s(Z_1,\ldots,Z_m;\PP_\kk^n)\right)$ in terms of the projective degrees of $\Gamma$.

\begin{corollary}[{cf.~\cite[Proposition 3.1]{ALUFFI_COMPUTING}}]
	\label{cor_computing}
	Assume \autoref{setup_computing} and the notations above.
	Then we have the equality
	$$
	\iota_*\left(s\left(Z_1,\ldots,Z_m;\PP_\kk^n\right) \right) \;=\; \left(1 - \sum_{i_1=0}^{r_1}\cdots\sum_{i_m=0}^{r_m} \frac{d_{i_1,\ldots,i_m}(\Phi)\, H^{i_1+\cdots+i_m}\,t_1^{i_1}\cdots t_m^{i_m}}{\left(1+d_1Ht_1\right)^{i_1+1}\cdots \left(1+d_mHt_m\right)^{i_m+1}}\right) \,\smallfrown\, \left[\PP_\kk^n\right].
	$$
\end{corollary}
\begin{proof}
	Let $\underline{Z}=Z_1,\ldots,Z_m$.
	Due to \autoref{thm_blow_up_form} and the equality $E_i = d_iH-K_i$, we have 
	$$
	\iota_*\left(s\left(\underline{Z};\PP_\kk^n\right) \right) \;=\; \left[\PP_\kk^n\right] - b_*\left(\prod_{i=1}^m\frac{1}{\big(1+d_iHt_i-K_it_i\big)}\, \smallfrown\, [\mathcal{B}]\right). 
	$$
	Then we can make the following simple algebraic manipulations
	\begin{align*}
		\prod_{i=1}^m\frac{1}{\big(1+d_iHt_i-K_it_i\big)} \,\smallfrown\, [\mathcal{B}] &\;=\;  \prod_{i=1}^m \left(\sum_{\ell=0}^\infty {(-1)}^\ell \left(d_iHt_i-K_it_i\right)^\ell\right) \,\smallfrown\, [\mathcal{B}] \\
		&\;=\;  \prod_{i=1}^m \left(\sum_{\ell=0}^\infty {(-1)}^\ell \sum_{k=0}^{\ell} \binom{\ell}{k} \left(d_iHt_i\right)^{\ell-k}\left(-K_it_i\right)^k\right) \,\smallfrown\, [\mathcal{B}] \\
		&\;=\;  \prod_{i=1}^m \left(\sum_{k=0}^\infty  \left(\sum_{\ell=k}^{\infty}{(-1)}^{\ell-k} \binom{\ell}{k} \left(d_iHt_i\right)^{\ell-k}\right)K_i^kt_i^k\right) \,\smallfrown\, [\mathcal{B}] \\
		&\;=\;  \prod_{i=1}^m \left(\sum_{k=0}^\infty  \left(\sum_{j=0}^{\infty}{(-1)}^{j} \binom{j+k}{k} \left(d_iHt_i\right)^{j}\right)K_i^kt_i^k\right) \,\smallfrown\, [\mathcal{B}] \\
		&\;=\;  \prod_{i=1}^m \left(\sum_{k=0}^\infty  \frac{K_i^kt_i^k}{\left(1+d_iHt_i\right)^{k+1}}\right) \,\smallfrown\, [\mathcal{B}] \\
		&\;=\; \sum_{i_1,\ldots,i_m\in \NN} \frac{K_1^{i_1}\cdots K_m^{i_m}t_1^{i_1}\cdots t_m^{i_m}}{\left(1+d_1Ht_1\right)^{i_1+1}\cdots \left(1+d_mHt_m\right)^{i_m+1}} \,\smallfrown\, [\mathcal{B}].
	\end{align*}
	By the projection formula and the fact that $K_j^{i_j} \cdot [\mathcal{B}] =0$ for $i_j > r_j$, we obtain 
	$$
	\iota_*\left(s\left(\underline{Z};\PP_\kk^n\right) \right) \;=\; \left(1 - \sum_{i_1=0}^{r_1}\cdots\sum_{i_m=0}^{r_m} \frac{d_{i_1,\ldots,i_m}(\Phi)\, H^{i_1+\cdots+i_m}\,t_1^{i_1}\cdots t_m^{i_m}}{\left(1+d_1Ht_1\right)^{i_1+1}\cdots \left(1+d_mHt_m\right)^{i_m+1}}\right) \;\smallfrown\; \left[\PP_\kk^n\right].
	$$
	So the proof of the corollary is complete.
\end{proof}

\begin{remark}
	Notice that the formula of \autoref{cor_computing} can be rewritten as 
	$$
	\iota_*\left(s\left(Z_1,\ldots,Z_m;\PP_\kk^n\right) \right) \;=\; \left(1 - \frac{Q(t_1,\ldots,t_m)}{\left(1+d_1Ht_1\right)^{r_1+1}\cdots \left(1+d_mHt_m\right)^{r_m+1}}\right) \,\smallfrown\, \left[\PP_\kk^n\right]
	$$
	where $Q(\ttt)=Q(t_1,\ldots,t_m)$ is the formal polynomial
	$$
	Q(\ttt) \;=\;  \sum_{i_1=0}^{r_1}\cdots\sum_{i_m=0}^{r_m} d_{i_1,\ldots,i_m}(\Phi)\, H^{i_1+\cdots+i_m}\,t_1^{i_1}\cdots t_m^{i_m}\left(1+d_1Ht_1\right)^{r_1-i_1}\cdots \left(1+d_mHt_m\right)^{r_m-i_m}.
	$$
This already gives an indication of the rationality of mixed Segre zeta functions (see \autoref{thm_rationality}). 
Indeed, notice that the formal polynomial $Q(\ttt)$ will not change if we consider $Z_1,\ldots,Z_m$ in a larger projective space $\PP_\kk^N$ by extending the ideals $I_1,\ldots,I_m$ to $\kk[x_0,\ldots,x_n,x_{n+1},\ldots,x_N]$.
\end{remark}

\section{Functoriality of mixed Segre classes under projections}
\label{sect_funct_Segre}
	
In this section, we study the pull-back of mixed Segre classes under certain rational maps. 
This technical consideration will be necessary to show the rationality of mixed Segre zeta functions in the subsequent  \autoref{sect_mixed_Seg_zeta}.

\begin{setup}
    \label{setup_functorial}
    Let $X$ and $Y$ be two irreducible varieties over a field $\kk$.	
    Let $\phi : X \dashrightarrow Y$ be a dominant rational map, $U \subset X$ be the domain of definition of $\phi$, and  $f : U \rightarrow Y$ be the associated morphism.
    Let $N = \dim(X)$ and $n = \dim(Y)$.
    Let $B_\phi \subset X$ be a closed subscheme such that the blow-up of $X$ along $B_\phi$ 
\begin{equation*}
		\begin{tikzpicture}[baseline=(current  bounding  box.center)]
			\matrix (m) [matrix of math nodes,row sep=3em,column sep=4.5em,minimum width=2em, text height=1.5ex, text depth=0.25ex]
			{
				 & \Gamma_\phi = \Bl_{B_\phi}(X) &  \\
				X & & Y\\
			};
			\path[-stealth]
			(m-1-2) edge node [above] {$\pi_1$} (m-2-1)
			(m-1-2) edge node [above] {$\pi_2$} (m-2-3);
			\draw[->,dashed] (m-2-1)--(m-2-3) node [midway,above] {$\phi$};	
			;		
		\end{tikzpicture}	
	\end{equation*}
	resolves the indeterminacies of $\phi$ and assume that the projection $\pi_2 \colon \Gamma_\phi \rightarrow Y$ is a flat morphism.
	Thus $f : U \cong \pi_1^{-1}(U) \rightarrow Y$ is also a flat morphism.
We say that \emph{$\phi$ is a projection with base locus $B_\phi$}.
Let $m \ge 1$ be a positive integer. 
	For each $1 \le i \le m$, let $r_i \ge 0$ be a nonnegative integer, $\EE_i$ be a vector bundle on $X$ of rank $r_i+1$, and $\FF_i$ be a vector bundle on $Y$ of rank $r_i+1$.
	We assume that $\EE_i$ and $\FF_i$ are \emph{compatible} with respect to $\phi$; this means that we get the following isomorphism
\begin{equation*}
    {\EE_i\mid}_{U}  \;\cong\; f^*\FF_i
\end{equation*}    
by restricting to $U$.
We also assume the existence of a  nonzero section $s_{\EE_i} \in \HH^0(X, \EE_i)$ of $\EE_i$ and a nonzero section $s_{\FF_i} \in \HH^0(Y, \FF_i)$ of $\FF_i$ with the \emph{compatibility} that the following diagram commutes
\begin{equation*}
		\begin{tikzpicture}[baseline=(current  bounding  box.center)]
			\matrix (m) [matrix of math nodes,row sep=3em,column sep=4.5em,minimum width=2em, text height=1.5ex, text depth=0.25ex]
			{
			 	\fSpec_U\left( \Sym\left( \left({\EE_i\mid}_U\right)^\vee \right) \right)& &  \fSpec_U \left( \Sym \left( \left(f^*\FF_i\right)^\vee \right) \right)\\
				 & U & \\
			};
			\path[-stealth]
			(m-2-2) edge node [below] {${s_{\EE_i}|}_U\;$} (m-1-1)
			(m-2-2) edge node [below] {$\;\quad f^*{s_{\FF_i}}$} (m-1-3)
			;		
			\draw[<->] (m-1-1)--(m-1-3) node [midway,above] {$\cong$};	
		\end{tikzpicture}	
\end{equation*}
	when restricting to $U$.
    Let $Z_i = Z(s_{\EE_i}) \subset X$ be the zero-scheme of $s_{\EE_i}$ and $W_i = Z(s_{\FF_i}) \subset Y$ be the zero-scheme of $s_{\FF_i}$.
    Let $Z \subset X$ and $W \subset Y$ be the closed subschemes defined by the ideal sheaves $\mathcal{I}_Z = \mathcal{I}_{Z_1} \cdots \mathcal{I}_{Z_m} \subset \OO_X$ and $\mathcal{I}_W = \mathcal{I}_{W_1} \cdots \mathcal{I}_{W_m} \subset \OO_Y$.
    In this subsection, we shall consider the immersion morphisms of several subschemes; thus for any subscheme $S' \subset S$ of a scheme $S$, we shall denote by $j_{S'} : S' \hookrightarrow S$ the corresponding immersion.
\end{setup}

We now present a notion of \emph{join} for Chow classes that was considered by Aluffi \cite[\S 3]{aluffi2017segre}.

\begin{definition}
	\label{def_join}
Let $C \in A_*(Y)$ be a Chow class on $Y$.
The \emph{join} of $C$ with respect to the base locus $B_\phi$ is the class given by 
$$
C \, \vee \, B_\phi \;:=\; {\pi_1}_*\left(\pi_2^*\left(C\right)\right) \;\in\; A_{*}(X).
$$
For any formal polynomial $C(\ttt) = \sum_{\alpha \in \NN^m} C_\alpha\, t_1^{\alpha_1}\cdots t_m^{\alpha_m} \in A_*(Y)\left[t_1,\ldots,t_m\right]$ in the variables $t_1,\ldots,t_m$ and with coefficients in $A_*(Y)$, we set 
$$
C(\ttt) \, \vee \, B_\phi \;:=\; \sum_{\alpha \in \NN^m} \left(C_\alpha \vee B_\phi\right)\, t_1^{\alpha_1}\cdots t_m^{\alpha_m} \in A_*(X)\left[t_1,\ldots,t_m\right].
$$
\end{definition}

\begin{remark}
\label{rem_uniq_join}
Let $C \in A_d(Y)$ be a class of dimension $d$.
We point out that if $\dim(B_\phi) < d+N-n$, then the exact sequence of Chow groups for an open embedding 
$$
A_{d+N-n}(B_\phi) \;\xrightarrow{\;{j_{B_\phi}}_*\;}\; A_{d+N-n}(X) \;\xrightarrow{\;j_{U}^*\;}\; A_{d+N-n}(U) \;\rightarrow\; 0
$$
(see \cite[\S 1.8]{FULTON_INTERSECTION_THEORY}) implies that the join $C \,\vee \, B_\phi$ is the \emph{unique} class in $A_{d+N-n}(X)$ whose restriction to $U$ coincides with the pull-back of $C$ via $f$; that is,
$$
j_U^*\left(C \,\vee \, B_\phi\right) \;=\; f^*\left(C\right).
$$
This basic observation will play an important role in our approach.
The inequality $\dim(B_\phi) < d+N-n$ will hold in the situations we consider later.
\end{remark}	

Let $\PP(\EE_i) := \fProj_X\left(\Sym(\EE_i^\vee)\right)$ and $\PP(\FF_i) := \fProj_Y\left(\Sym(\FF_i^\vee)\right)$ be the corresponding projective bundles (we use the notation of Fulton \cite{FULTON_INTERSECTION_THEORY} with respect to projective bundles). 
We consider the following products of projective bundles 
$$
\PP\left(\underline{\EE}\right) \;:=\; \PP(\EE_1) \times_X \cdots \times_X \PP(\EE_m)
\quad \text{ and } \quad
\PP\left(\underline{\FF}\right) \;:=\; \PP(\FF_1) \times_Y \cdots \times_Y \PP(\FF_m).
$$
Let $q_1 : \PP\left(\underline{\EE}\right) \rightarrow X$ and $q_2 : \PP\left(\underline{\FF}\right) \rightarrow Y$ be the natural projections. 
The joint blow-ups $\Bl_{Z_1,\ldots,Z_m}(X)$ and $\Bl_{W_1,\ldots,W_m}(Y)$ can be seen as (the closures) of the images of the induced rational maps 
$$
s_{\underline{\mathcal{E}}}=(s_{\EE_1}, \ldots,s_{\EE_m}) \;:\; X  \;\dashrightarrow\; \PP\left(\underline{\EE}\right) = \PP(\EE_1) \times_X \cdots \times_X \PP(\EE_m)
$$
and 
$$
s_{\underline{\mathcal{F}}}=(s_{\FF_1}, \ldots,s_{\FF_m}) \;:\; Y\;\dashrightarrow\; \PP\left(\underline{\FF}\right) = \PP(\FF_1) \times_Y \cdots \times_Y \PP(\FF_m),
$$
respectively.

Our compatibility assumptions yield the following fibre square
\begin{equation}
\label{eq_fibre_sq_f_g}
		\begin{tikzpicture}[baseline=(current  bounding  box.center)]
			\matrix (m) [matrix of math nodes,row sep=3em,column sep=7em,minimum width=2em, text height=1.5ex, text depth=0.25ex]
			{
			 	\UU\;:=\; \PP\left({\underline{\EE}|}_U\right) \;\cong\; \PP\left(f^*\left(\underline{\FF}\right)\right) &  \PP(\underline{\FF})\\
				  U & Y.\\
			};
			\path[-stealth]
			(m-1-1) edge node [left] {$q_1$} (m-2-1)
                (m-1-2) edge node [right] {$q_2$} (m-2-2)
                (m-1-1) edge node [above] {$g$}(m-1-2)
                (m-2-1) edge node [above] {$f$}(m-2-2)
			;
		\end{tikzpicture}	
\end{equation}
where ${\underline{\EE}|}_U = {\EE_1|}_U,\ldots,{\EE_m|}_U$, $f^*\left(\underline{\FF}\right) = f^*(\FF_1),\ldots,f^*(\FF_m)$ and $g : \UU \rightarrow \PP(\underline{\FF})$ is also flat.
By an abuse of notation, we also denote by $q_1 : \mathbb{U} \rightarrow U$ the restriction of the projection $q_1 :\PP\left(\underline{\EE}\right) \rightarrow X$ to $\mathbb{U}$.
Thus we obtain a rational map
$$
\Psi \;:\; \PP(\underline{\EE}) \;\dashrightarrow\; \PP(\underline{\FF}) 
$$
with base locus $B_\Psi = B_\phi \times_X \PP(\underline{\EE})$.
The blow-up of $\PP(\underline{\EE})$ along $B_\Psi$ is given by 
\begin{equation*}
\Gamma_\Psi \;\cong\; \Gamma_\phi \times_X \PP(\underline{\EE})
\end{equation*}
because 
$$
\Rees_{\OO_X}\left(\mathcal{I}_{B_\phi}\right) \otimes_{\OO_X} \OO_{\PP(\underline{\EE})} \;\xrightarrow{\;\cong\;}\; \Rees_{\OO_{\PP(\underline{\EE})}}\left(\mathcal{I}_{B_\phi}\OO_{\PP(\underline{\EE})}\right) \;\cong\; \Rees_{\OO_{\PP(\underline{\EE})}}\left(\mathcal{I}_{B_\Psi}\right).
$$
By combining everything, we obtain the following commutative diagram that encodes all the data
\begin{equation*}
\begin{tikzpicture}[baseline=(current  bounding  box.center)]
			\matrix (m) [matrix of math nodes,row sep=2.5em,column sep=4.5em,minimum width=2em, text height=1.5ex, text depth=0.25ex]
			{
			  \PP(\underline{\EE}) & & \PP(\underline{\FF}) \\ 
              & \Gamma_\Psi = \Bl_{B_\Psi}(\PP(\underline{\EE})) &  \\
                 & \Gamma_\phi = \Bl_{B_\phi}(X) &  \\
				X & & Y.\\
			};
			\path[-stealth]
	(m-3-2) edge node [above] {$\pi_1$} (m-4-1)
	(m-3-2) edge node [above] {$\pi_2$} (m-4-3)
        (m-2-2) edge node [above] {$p_1$} (m-1-1)
	(m-2-2) edge node [above] {$p_2$} (m-1-3)
        (m-2-2) edge node [left] {$\rho$} (m-3-2)
        (m-1-1) edge node [left] {$q_1$} (m-4-1)
        (m-1-3) edge node [right] {$q_2$} (m-4-3)
            ;
		\draw[->,dashed] (m-4-1)--(m-4-3) node [midway,above] {$\phi$};	
		\draw[->,dashed] (m-1-1)--(m-1-3) node [midway,above] {$\Psi$};	
		\end{tikzpicture}	
\end{equation*}
It only remains to show that the trapezoid on the right commutes; however, the equality $q_2 \circ p_2 = \pi_2 \circ \rho$ holds since these two compositions of morphisms agree on the complement of the exceptional divisor of $\Gamma_\Psi$ (see \cite[Exercise II.4.2]{HARTSHORNE}).

\begin{lemma}
    \label{lem_join_blow_ups}
    Consider the Chow classes $\left[\Bl_{Z_1,\ldots,Z_m}(X)\right] \in A_N\left(\PP(\underline{\EE})\right)$ and $\left[\Bl_{W_1,\ldots,W_m}(Y)\right] \in A_n\left(\PP(\underline{\FF})\right)$.
    Assume that $\dim\left(B_\phi\right) < N - \sum_{i=1}^m r_i$.
    Then $\left[\Bl_{Z_1,\ldots,Z_m}(X)\right]$ is the unique class in $A_N\left(\PP(\underline{\EE})\right)$ such that 
    $$
    j_\UU^*\big(\left[\Bl_{Z_1,\ldots,Z_m}(X)\right]\big) \;=\; g^*\big(\left[\Bl_{W_1,\ldots,W_m}(Y)\right]\big).
    $$
\end{lemma}
\begin{proof}
Here we consider the rational map $\Psi : \PP(\underline{\EE}) \dashrightarrow \PP(\underline{\FF})$.
By our compatibility assumptions (see \autoref{setup_functorial}), we have
$$
j_\UU^*\big(\left[\Bl_{Z_1,\ldots,Z_m}(X)\right]\big) \;=\; g^*\big(\left[\Bl_{W_1,\ldots,W_m}(Y)\right]\big).
$$
Since $\dim\left(B_\Psi\right) = \dim\left(B_\phi\right) + \sum_{i=1}^mr_i < N$, the uniqueness claim follows by utilizing the exact sequence of Chow groups for an open embedding \cite[\S 1.8]{FULTON_INTERSECTION_THEORY}.
\end{proof}

The total Chern class of $\underline{\EE}=\EE_1,\ldots,\EE_m$ is given by 
$$
c_\ttt(\underline{\EE}) \;:=\; c_{t_1}(\EE_1) \cdots c_{t_m}(\EE_m) \;=\; \sum_{i_1,\ldots,i_m \in \NN} c_{i_1}(\EE_1) \cdots c_{i_m}(\EE_m) \, t_1^{i_1}\cdots t_m^{i_m} 
$$
and the total Segre class is given by 
$$
s_\ttt(\underline{\EE}) \;:=\; s_{t_1}(\EE_1) \cdots s_{t_m}(\EE_m) \;=\; \sum_{i_1,\ldots,i_m \in \NN} s_{i_1}(\EE_1) \cdots s_{i_m}(\EE_m) \, t_1^{i_1}\cdots t_m^{i_m}. 
$$
We regard both $c_\ttt(\underline{\mathcal{E}})$ and $s_\ttt(\underline{\mathcal{E}})$ as operators over the formal polynomials in $A_*(X)[t_1,\ldots,t_m]$.
Similarly, we consider the total Chern class $c_\ttt(\underline{\mathcal{F}})$ and the total Segre class $s_\ttt(\underline{\mathcal{F}})$ of $\underline{\FF}=\FF_1,\ldots,\FF_m$.
We are interested on the following total classes
$$
L_\ttt\left(\underline{Z};X\right) \;:=\; c_\ttt(\underline{\EE}) \,\smallfrown\, {q_1}_*\left( \frac{1}{c_{t_1}\left(\OO_{\PP(\underline{\EE})}(-\ee_1)\right)\cdots c_{t_m}\left(\OO_{\PP(\underline{\EE})}(-\ee_m)\right)} \,\smallfrown\, \left[\Bl_{Z_1,\ldots,Z_m}(X)\right]\right) 
$$
and 
$$
L_\ttt\left(\underline{W};Y\right) \;:=\; c_\ttt(\underline{\FF}) \,\smallfrown\, {q_2}_*\left( \frac{1}{c_{t_1}\left(\OO_{\PP(\underline{\FF})}(-\ee_1)\right)\cdots c_{t_m}\left(\OO_{\PP(\underline{\FF})}(-\ee_m)\right)} \,\smallfrown\, \left[\Bl_{W_1,\ldots,W_m}(Y)\right]\right).
$$
In \cite{aluffi2017segre}, when $m=1$, these classes are referred to as the shadow of the blow-up.
\begin{remark}
    \label{rem_formula_L}
    As a direct consequence of \autoref{thm_blow_up_form}, we obtain 
    $$    {j_{Z}}_*\left(s\left(Z_1,\ldots,Z_m;X\right)\right) \;=\; [X] - s_\ttt(\underline{\EE}) \smallfrown L_\ttt\left(\underline{Z};X\right).
    $$
    and 
    $$    {j_{W}}_*\left(s\left(W_1,\ldots,W_m;Y\right)\right) \;=\; [Y] - s_\ttt(\underline{\FF}) \smallfrown L_\ttt\left(\underline{W};Y\right).
    $$
\end{remark}

The next lemma yields convenient bounds for $L_\ttt\left(\underline{Z};X\right)$ and $L_\ttt\left(\underline{W};Y\right)$ (it justifies the apriori naive algebraic manipulation of \autoref{rem_formula_L}, where we multiply by a total Chern class and then by a total Segre class).

\begin{lemma}
\label{lem_bounds_L}
We have the following expressions bounded by the ranks $r_1,\ldots,r_m${\rm:}
\begin{enumerate}[\rm (i)]
\item There are classes $C_\alpha \in A_{N-\sum_{i=1}^mr_i+\sum_{i=1}^m\alpha_i}(X)$ with $\alpha = (\alpha_1,\ldots,\alpha_m) \in \NN^m$ such that 
$$
L_\ttt(\underline{Z}; X) \;=\; \sum_{0 \le \alpha_i\le r_i} \, C_\alpha \, t_1^{r_1-\alpha_1} \cdots t_m^{r_m-\alpha_m}.
$$
\item There are classes $D_\alpha \in A_{n-\sum_{i=1}^mr_i+\sum_{i=1}^m\alpha_i}(Y)$ with $\alpha = (\alpha_1,\ldots,\alpha_m) \in\NN^m$ such that 
$$
L_\ttt(\underline{W}; Y) \;=\; \sum_{0 \le \alpha_i\le r_i} \, D_\alpha \, t_1^{r_1-\alpha_1} \cdots t_m^{r_m-\alpha_m}.
$$
\end{enumerate}
\end{lemma}
\begin{proof}
(i) 
From the formula for the Chow groups of projective bundles \cite[Theorem 3.3]{FULTON_INTERSECTION_THEORY}, we have a unique expression 
$$
\big[\Bl_{Z_1,\ldots,Z_m}\big] \;=\; \sum_{0 \le \alpha_i\le r_i} c_1\left(\OO_{\PP(\underline{\EE})}(\ee_1)\right)^{\alpha_1}\cdots c_1\left(\OO_{\PP(\underline{\EE})}(\ee_m)\right)^{\alpha_m} \,\smallfrown\, q_1^*\left(C_\alpha\right)
$$
where $C_\alpha \in A_{N-\sum_{i=1}^mr_i+\sum_{i=1}^m\alpha_i}(X)$.
Thus we get the formula
\begin{align*}
L_\ttt\left(\underline{Z};X\right) \;=\; c_\ttt(\underline{\EE}) \,\smallfrown\, \sum_{0 \le \alpha_i\le r_i} 
 {q_1}_*\left( \sum_{j_i \ge \alpha_i} 
 \left(
 \prod_{i=1}^mc_1\left(\OO_{\PP(\underline{\EE})}(\ee_i)\right)^{j_i} t_i^{j_i-\alpha_i}
 \right) \,\smallfrown\, q_1^*\left(C_\alpha\right) \right).
\end{align*}
Since ${q_1}_*\big(c_1\left(\OO_{\PP(\underline{\EE})}(\ee_1)\right)^{j_1}\cdots c_1\left(\OO_{\PP(\underline{\EE})}(\ee_m)\right)^{j_m} \,\smallfrown\, q_1^*\left(C_\alpha\right)\big) = 0$ when $j_i < r_i$ for any $1 \le i \le m$ (see \cite[Proposition 3.1(a)(i)]{FULTON_INTERSECTION_THEORY}), we can write
\begin{align*}
L_\ttt\left(\underline{Z};X\right) \;&=\; c_\ttt(\underline{\EE}) \,\smallfrown\, \sum_{0 \le \alpha_i\le r_i} \frac{1}{t_1^{\alpha_1}\cdots t_m^{\alpha_m}}\;
 {q_1}_*\left(\frac{1}{\prod_{i=1}^mc_{t_i}\left(\OO_{\PP(\underline{\EE})}(-\ee_i)\right)} \,\smallfrown\, q_1^*\left(C_\alpha\right)\right) \\
 \;&=\;  \sum_{0 \le \alpha_i\le r_i} \frac{1}{t_1^{\alpha_1}\cdots t_m^{\alpha_m}}\;
 {q_1}_*\left(\frac{c_\ttt(q_1^*(\underline{\EE}))}{\prod_{i=1}^mc_{t_i}\left(\OO_{\PP(\underline{\EE})}(-\ee_i)\right)} \,\smallfrown\, q_1^*\left(C_\alpha\right)\right),
\end{align*}
where $q_1^*(\underline{\EE}) = q_1^*(\EE_1),\ldots,q_1^*(\EE_m)$.
Consider the quotient bundle $\mathcal{Q}_i := q_1^*(\EE_i) / \OO_{\PP(\underline{\EE})}(-\ee_i)$ on $\PP(\underline{\EE})$ of rank $r_i$.
Finally, \cite[Example 3.3.3]{FULTON_INTERSECTION_THEORY} and the Whitney sum formula give 
\begin{align*}
L_\ttt\left(\underline{Z};X\right) \;&=\;   \sum_{0 \le \alpha_i\le r_i} \frac{1}{t_1^{\alpha_1}\cdots t_m^{\alpha_m}}\;
 {q_1}_*\big(c_{t_1}(\mathcal{Q}_1) \cdots c_{t_m}(\mathcal{Q}_m) \,\smallfrown\, q_1^*\left(C_\alpha\right)\big) \\
 \;&=\;  \sum_{0 \le \alpha_i\le r_i} \, C_\alpha \, t_1^{r_1-\alpha_1} \cdots t_m^{r_m-\alpha_m}.
\end{align*}

(ii) This part follows verbatim to part (i).
\end{proof}

The key result in this section is the following proposition. 

\begin{proposition}
    \label{prop_join_L}
    If $\dim\left(B_\phi\right) < N - \sum_{i=1}^m r_i$, then 
    $$
    L_\ttt(\underline{Z};X) \;=\; L_\ttt(\underline{W}; Y) \,\vee\, B_\phi.
    $$
\end{proposition}
\begin{proof}
    Due to the assumption $\dim\left(B_\phi\right) < N - \sum_{i=1}^m r_i$, the bounds of \autoref{lem_bounds_L}, and \autoref{rem_uniq_join}, it suffices to show that 
    $$
    j_U^*\left(L_\ttt(\underline{Z};X)\right) \;=\; f^*\left(L_\ttt(\underline{W};Y)\right). 
    $$
    Let $\mathcal{B}=\Bl_{Z_1,\ldots,Z_m}(X)$ and $\widetilde{\mathcal{B}}=\Bl_{W_1,\ldots,W_m}(Y)$.
We utilize the fibre square \autoref{eq_fibre_sq_f_g} and \cite[Proposition 1.7]{FULTON_INTERSECTION_THEORY} to compute
\begin{align*}
f^*\left(L_\ttt\left(\underline{W};Y\right)\right) &\;=\; f^*\left(c_\ttt(\underline{\FF}) \,\smallfrown\, {q_2}_*\left( \frac{1}{\prod_{i=1}^mc_{t_i}\left(\OO_{\PP(\underline{\FF})}(-\ee_i)\right)} \,\smallfrown\, \left[\widetilde{\mathcal{B}}\right]\right)\right) \\
&\;=\; c_\ttt({\underline{\EE}|}_U) \,\smallfrown\, {q_1}_*\left(g^*\left( \frac{1}{\prod_{i=1}^mc_{t_i}\left(\OO_{\PP(\underline{\FF})}(-\ee_i)\right)} \,\smallfrown\, \left[\widetilde{\mathcal{B}}\right]\right)\right) \\
&\;=\; c_\ttt({\underline{\EE}|}_U) \,\smallfrown\, {q_1}_*\left(\frac{1}{\prod_{i=1}^mc_{t_i}\left(\OO_{\PP\left({\underline{\EE}|}_U\right)}(-\ee_i)\right)} \,\smallfrown\, j_\UU^*\left(\left[\mathcal{B}\right]\right)\right) \quad\; \text{(see \autoref{lem_join_blow_ups})}\\
&\;=\; j_U^*\left(L_\ttt(\underline{Z};X)\right).
\end{align*}
This completes the proof.
\end{proof}

The main result of this section is the following theorem.

\begin{theorem}
	\label{thm_pullback_Seg}
    Assume \autoref{setup_functorial} and the notations above.
    If $\dim\left(B_\phi\right) < N - m - \sum_{i=1}^m r_i$, then we have the equality
    $$
    {j_Z}_*\left(s(Z_1,\ldots,Z_m;X)\right) \;=\; s_\ttt\left(\underline{\EE}\right) \,\smallfrown\, \left( \Big( c_\ttt\left(\underline{\FF}\right) \,\smallfrown\, {j_W}_*\left(s(W_1,\ldots,W_m;Y)\right) \Big)  \,\vee\, B_\phi \right).
    $$
\end{theorem}
\begin{proof}
    Let $\underline{Z} = Z_1,\ldots,Z_m$ and $\underline{W} = W_1,\ldots,W_m$.
    From \autoref{rem_formula_L} and \autoref{prop_join_L}, we obtain
    \begin{align*}
        {j_Z}_*\left(s(\underline{Z};X)\right) &\;=\; [X] - s_\ttt(\underline{\EE}) \smallfrown L_\ttt\left(\underline{Z};X\right)\\
        &\;=\; [X] - s_\ttt(\underline{\EE}) \smallfrown \left(L_\ttt(\underline{W};Y) \,\vee\; B_\phi\right)\\
        &\;=\; [X] - s_\ttt(\underline{\EE}) \smallfrown \left(\Big( c_\ttt(\underline{\FF}) \smallfrown [Y] - c_\ttt(\underline{\FF}) \smallfrown {j_W}_*\left(s(\underline{W};Y)\right) \Big) \,\vee\; B_\phi\right).
    \end{align*}
    Therefore, to conclude the proof it suffices to show that 
    $$
    c_\ttt(\underline{\EE}) \smallfrown [X] \;=\; \big(c_\ttt(\underline{\FF}) \smallfrown [Y]\big) \,\vee\, B_\phi;
    $$
    however, this equality already follows from the upper bound $\dim\left(B_\phi\right) < N  - \sum_{i=1}^m (r_i+1) = N - \sum_{i=1}^{m}\rank(\EE_i)$, our compatibility assumptions (see \autoref{setup_functorial}) and \autoref{rem_uniq_join}.
\end{proof}

\section{Existence and rationality of mixed Segre zeta functions}	
\label{sect_mixed_Seg_zeta}

We are now ready to prove the existence and rationality of mixed Segre zeta functions. 
Throughout this section,   we utilize the developments of the previous \autoref{sect_funct_Segre} in the following setting. 

\begin{setup}
	\label{setup_mixed_Segre_zeta}
	Let $R = \kk[x_0,\ldots,x_n]$ be a polynomial ring over a field $\kk$.
	Let $I_1,\ldots,I_m \subset R$ be homogeneous ideals in $R$.
	For each $1 \le i \le m$, let $f_{i, 0}, f_{i, 1}, \dotsc, f_{i, r_i} \in R$ be homogeneous generators of $I_i$, with the degree of $f_{i, j}$ equal to $d_{i,j} = \deg\left(f_{i,j}\right)$.
	Let $Z_i$ be the closed subscheme defined by $I_i \subset R$ in $\PP_\kk^n = \Proj(R)$. 
	Denote by $\iota : Z \hookrightarrow \PP_\kk^n$ the closed subscheme defined by the product ideal $I_1\cdots I_m \subset R$.
	For each integer $N \ge n$, we consider the following objects:
	\begin{itemize}[--]
		\item $R^N = \kk[x_0,\ldots,x_n, \ldots, x_N]$ a polynomial ring containing $R$ and $\PP_\kk^N = \Proj(R^N)$.
		\item $Z_i^N \subset \PP_\kk^N$ the closed subscheme defined by the extension of $I_i$ to $R^N$.
		\item $\iota_N : Z^N \hookrightarrow \PP_\kk^N$ the closed subscheme defined by the extension of $I_1\cdots I_m$ to $R^N$.
	\end{itemize}
\end{setup}

By \autoref{lem_existence} below, there is a formal power series with integer coefficients 
$$
\zeta_{I_1, \ldots, I_m}(t_1,\ldots,t_m) \;=\; \sum_{i_1,\ldots,i_m \in \NN} \, a_{i_1,\ldots,i_m} \, t_1^{i_1}\cdots t_m^{i_m} \;\in\; \ZZ[\![t_1,\ldots,t_m]\!]
$$
such that 
$$
{\iota_N}_*\left(s^{i_1,\ldots,i_m}\left(Z_1^N,\ldots,Z_m^N;\, \PP_\kk^N\right)\right) \;=\; a_{i_1,\ldots,i_m} \, H^{i_1+\cdots+i_m} \,\smallfrown\, \left[\PP_\kk^N\right]
$$
for all $N \ge n$ and $i_1 + \cdots + i_m \le N$.

\begin{definition}
	We call $\zeta_{I_1, \ldots, I_m}(t_1,\ldots,t_m)$ the \emph{mixed Segre zeta function} of the ideals $I_1,\ldots,I_m$.	
\end{definition}

The main result of this section is the following theorem.

\begin{theorem}[Rationality of mixed Segre zeta functions]
	\label{thm_rationality}
	Assume \autoref{setup_mixed_Segre_zeta}.
	Then we can write 
		$$
		\zeta_{I_1, \ldots, I_m}(t_1,\ldots,t_m) \;\;=\;\; \frac{P(t_1,\ldots,t_m)}{\;\prod_{1 \le i \le m} \left(1+d_{i,0}t_i\right)\left(1+d_{i,1}t_i\right)\cdots\left(1+d_{i,r_i}t_i\right)\;}
		$$
		where $P(t_1,\ldots,t_m) \in \NN[t_1,\ldots,t_m]$ is a polynomial with nonnegative integer coefficients.
\end{theorem}

We first establish the notation required for the proof of the above theorem.
Let $N \ge n$.
Consider the vector bundles $\underline{\mathcal{E}}^N = \mathcal{E}_1^N,\ldots,\mathcal{E}_m^N$ on $\PP_\kk^N$ with 
$$
\mathcal{E}_i^N \;:=\; \OO_{\PP_\kk^N}\left(d_{i,0}\right) \,\oplus\, \OO_{\PP_\kk^N}\left(d_{i,1}\right) \,\oplus\, \cdots \,\oplus\, \OO_{\PP_\kk^N}\left(d_{i,r_i}\right).
$$
Since $d_{i,j} > 0$, each vector bundle $\EE_i^N$ is generated by global sections. 
We have that $s_i^N = (f_{i,0},\ldots,f_{i,r_i}) \in \HH^0(\PP_\kk^N, \mathcal{E}_i^N)$ is a section of $\mathcal{E}_i^N$.
The joint blow-up 
$$
b^N \;:\; \mathcal{B}^N = \Bl_{Z_1^N,\ldots,Z_m^N}(\PP_\kk^N) \;\rightarrow\; \PP_\kk^N
$$ 
can be seen as (the closure) of the image of the induced rational map 
$$
s^N = (s_{1}^N, \ldots,s_{m}^N) \;:\; \PP_\kk^N  \;\dashrightarrow\; \PP\big(\underline{\EE}^N\big) \;=\; \PP(\EE_1^N) \times_{\PP_\kk^N} \cdots \times_{\PP_\kk^N} \PP(\EE_m^N).
$$
Let $q_N : \PP(\underline{\mathcal{E}}^N) \rightarrow \PP_\kk^N$ be the natural projection.
When $N = n$, we simply write $\mathcal{E}_i := \mathcal{E}_i^n$, $\underline{\mathcal{E}} := \underline{\mathcal{E}}^n$, $q := q^n$, $s_i := s_i^n$, $s:=s^n$, $b := b^n$, and $\mathcal{B}:=\mathcal{B}^n$.
Let $E_i$ be the pull-back to $\mathcal{B} = \Bl_{Z_1,\ldots,Z_m}(\PP_\kk^n)$ of the exceptional divisor of the blow-up $\Bl_{Z_i}(\PP_\kk^n)$.

The rational map 
$$
\phi_N : \PP_\kk^N \dashrightarrow \PP_\kk^n, \qquad \left(x_0:\cdots:x_n:x_{n+1}:\cdots:x_N\right) \mapsto \left(x_0:\cdots:x_n\right)
$$ 
is a projection in the sense of \autoref{setup_functorial}.
The base locus of $\phi_N$ is equal to $B_{\phi_N} = V(x_0,\ldots,x_n)\cong \PP_\kk^{N-n-1} \subset \PP_\kk^N$.
Let $U_N := \PP_\kk^N \setminus B_{\phi_N}$ and $f_N : U_N \rightarrow \PP_\kk^n$ be the associated morphism. 
Then we can check that the projection $\phi_N : \PP_\kk^N \dashrightarrow \PP_\kk^n$ satisfies all the compatibility assumptions of \autoref{setup_functorial}.
For instance, we have the compatibility isomorphism
$$
{\mathcal{E}_i^{N}\mid}_{U_N} \;\cong\; f_N^*\left(\EE_i\right).
$$
In terms of $\phi_N : \PP_\kk^N \dashrightarrow \PP_\kk^n$ and the base locus $B_{\phi_N}$, the join operation introduced in \autoref{def_join} acts in the simplest possible way.
More precisely, for any formal polynomial  $$
C(\ttt) \;=\; \left(\sum_{\substack{\alpha \in \NN^m\\0\le  k  \le n}} c_{\alpha,k}\, t_1^{\alpha_1}\cdots t_m^{\alpha_m} H^k\right) \,\smallfrown\, \left[\PP_\kk^n\right] \;\;\in\;\; A^*(\PP_\kk^n)\left[\ttt\right] \;\cong\; \frac{\ZZ\left[H, \ttt\right]}{\left(H^{n+1}\right)}
$$ the join with respect to $B_{\phi_N}$ is equal to
$$
C(\ttt) \,\vee\; B_{\phi_N} \;=\; \left(\sum_{\substack{\alpha \in \NN^m\\0\le  k  \le n}} c_{\alpha,k}\, t_1^{\alpha_1}\cdots t_m^{\alpha_m} H^k\right) \,\smallfrown\, \left[\PP_\kk^N\right] \;\;\in\;\; A^*(\PP_\kk^N)\left[\ttt\right] \;\cong\; \frac{\ZZ\left[H, \ttt\right]}{\left(H^{N+1}\right)}
$$ 
where $H$ denotes the class of a hyperplane on both $\PP_\kk^n$ and $\PP_\kk^N$.

The existence of $\zeta_{I_1, \ldots, I_m}(t_1,\ldots,t_m)$ follows from the following lemma.

\begin{lemma}[Existence of mixed Segre zeta functions]
	\label{lem_existence}
	For any $M > N \ge n$ and $i_1+\cdots+i_m \le N$, if we write 
	$$
	{\iota_N}_*\left(s^{i_1,\ldots,i_m}\left(Z_1^N,\ldots,Z_m^N;\, \PP_\kk^N\right)\right) \;=\; a_{i_1,\ldots,i_m} \, H^{i_1+\cdots+i_m} \,\smallfrown\, \left[\PP_\kk^N\right]
	$$
	and 
	$$
	{\iota_{M}}_*\left(s^{i_1,\ldots,i_m}\left(Z_1^{M},\ldots,Z_m^{M};\, \PP_\kk^{M}\right)\right) \;=\; b_{i_1,\ldots,i_m} \, H^{i_1+\cdots+i_m} \,\smallfrown\, \left[\PP_\kk^{M}\right], 
	$$
	then $a_{i_1,\ldots,i_m} = b_{i_1,\ldots,i_m}$.
\end{lemma}
\begin{proof}
	We may assume that $M=N+1$. 
	We see $\PP_\kk^N$ as the divisor on $\PP_\kk^{N+1}$ determined by the variable $x_{N+1}$ in $R^{N+1} = \kk[x_0,\ldots,x_{N+1}]$.
	Since each ideal $I_i$ does not involve the variable $x_{N+1}$, we obtain the following specialization of joint blow-ups
	$$
	q_{N+1}^*H \,\cdot\, \left[\mathcal{B}^{N+1}\right] \;=\; \left[\mathcal{B}^N\right].
	$$
	By utilizing \autoref{thm_blow_up_form} and the projection formula, we get 
	$$
	H \,\cdot\, {\iota_{N+1}}_*\left(s\left(Z_1^{N+1},\ldots,Z_m^{N+1}; \,\PP_\kk^{N+1}\right)\right) \;=\; {\iota_{N}}_*\left(s\left(Z_1^{N},\ldots,Z_m^{N}; \,\PP_\kk^{N}\right)\right).
	$$
	This implies the result of the lemma.
\end{proof}

Given a formal power series $F(\ttt) = F(t_1,\ldots,t_m) = \sum_{\alpha \in \NN^m} c_\alpha t_1^{\alpha_1}\cdots t_m^{\alpha_m} \in \ZZ[\![t_1,\ldots,t_m]\!]$ and a nonnegative integer $k \ge 0$, we define the \emph{$k$-truncation} as the sum 
$$
{\rm T}_k\left(F(\ttt)\right) \;:=\; \sum_{|\alpha| \le k} c_\alpha t_1^{\alpha_1}\cdots t_m^{\alpha_m} \;\in\; \ZZ[t_1,\ldots,t_m]
$$
of terms with total degree $\le k$.
Our main application of \autoref{thm_pullback_Seg} is given in the following lemma.

\begin{lemma}
	\label{lem_truncation}
	If $n \ge m+\sum_{i=1}^{m}r_i$, then the following statements hold for all $N \ge n$:
	\begin{enumerate}[\rm(i)]
		\item In $A^*(\PP_\kk^N)[t_1,\ldots,t_m]$, we have the equality
		$$
		{\iota_N}_*\left(s\left(Z_1^N,\ldots,Z_m^N;\, \PP_\kk^N\right)\right) \;=\; \frac{
			\Big(
			\prod_{i,j}\left(1+d_{i,j}\,H\,t_i\right) \cdot {\iota}_*\left(s\left(Z_1,\ldots,Z_m;\, \PP_\kk^n\right)\right)
			\Big) \;\vee\; B_{\phi_N}
		}{\;\prod_{i,j} \left(1+d_{i,j}\,H\,t_i\right)\;}.
		$$
		\item In $\ZZ[t_1,\ldots,t_m]$, we have the equality
		$$
		{\rm T}_N\left(\zeta_{I_1, \ldots, I_m}(\ttt)\right) \;=\; {\rm T}_N\left(
		\frac{
			{\rm T}_n\Big(
			\prod_{i,j} \left(1+d_{i,j}\,t_i\right)	\zeta_{I_1, \ldots, I_m}(\ttt)
			\Big)
		}{\prod_{i,j} \left(1+d_{i,j}\,t_i\right)}
		\right).	
		$$
	\end{enumerate}	
\end{lemma}
\begin{proof}
	(i) Since $\dim(B_{\phi_N}) = N-n-1$ and $n \ge m+\sum_{i=1}^{m}r_i$, we get $\dim(B_{\phi_N}) < N -m- \sum_{i=1}^{m}r_i$.
	Therefore \autoref{thm_pullback_Seg} applied to $\phi_N: \PP_\kk^N \dashrightarrow \PP_\kk^n$ gives 
	$$
	{\iota_{N}}_*\left(s(Z_1^N,\ldots,Z_m^N;\PP_\kk^N)\right) \;=\; s_\ttt(\underline{\EE}^N) \,\smallfrown\, \left( \Big( c_\ttt\left(\underline{\EE}\right) \,\smallfrown\, {\iota}_*\left(s(Z_1,\ldots,Z_m;\PP_\kk^n)\right) \Big)  \,\vee\, B_{\phi_N} \right).
	$$
	Furthermore, notice that 
	$$
	s_\ttt(\underline{\mathcal{E}}^N) \;=\; \frac{1}{\prod_{i,j} \left(1+d_{i,j}\,H\,t_i\right)} \;\in\; A^*(\PP_\kk^N)[\ttt] \quad\text{ and } \quad 
	c_\ttt(\underline{\mathcal{E}}) \;=\; \prod_{i,j} \left(1+d_{i,j}\,H\,t_i\right) \;\in\; A^*(\PP_\kk^n)[\ttt].
	$$
	So the result of this part follows.
	
	(ii) This part follows from part (i) and the existence of the mixed Segre zeta function $\zeta_{I_1, \ldots, I_m}(\ttt)$ (see \autoref{lem_existence}).
\end{proof}

Following Fulton \cite[Chapter 12]{FULTON_INTERSECTION_THEORY}, we denote by $A_*^{\ge}(\PP_\kk^n)$ the set of classes that can be represented by nonnegative cycles. 
Denote by $A_*^{\ge}(\PP_\kk^n)[t_1,\ldots,t_m]$ the set of formal polynomials in the variables $t_1,\ldots,t_m$ and with coefficients in $A_*^{\ge}(\PP_\kk^n)$.

\begin{lemma}
	\label{lem_nonnegative}
	We have that 
	$$
	\prod_{i,j}\left(1+d_{i,j}\,H\,t_i\right) \cdot {\iota}_*\left(s\left(Z_1,\ldots,Z_m;\, \PP_\kk^n\right)\right) \;\in\;  A_*^{\ge}(\PP_\kk^n)[t_1,\ldots,t_m]
	$$
	is nonnegative.
\end{lemma}
\begin{proof}
	Let $\underline{Z} = Z_1,\ldots,Z_m$.
	Consider the quotient bundle $\mathcal{Q}_i := q^*(\EE_i) / \OO_{\PP(\underline{\EE})}(-\ee_i)$ on $\PP(\underline{\EE})$ of rank $r_i$.	
	By utilizing \autoref{rem_seg_fulton} and the projection and Whitney sum formulas, we obtain
	\begin{align*}
	c_\ttt(\underline{\mathcal{E}}) \,\smallfrown\, {\iota}_*\left(s\left(\underline{Z};\, \PP_\kk^n\right)\right) &\;=\;
	\sum_{k=1}^m\,
	c_\ttt(\underline{\mathcal{E}}) \,\smallfrown\, q_*\left(
	\frac{1}{\prod_{i=k}^mc_{t_i}\left(\OO_{\PP(\underline{\EE})}(-\ee_i)\right)} \,\smallfrown\,
	\left[E_k\right]t_k\right)\\
	&\;=\;
	\sum_{k=1}^m q_*\Big(c_{t_1}(q^*(\EE_1)) \cdots c_{t_{k-1}}(q^*(\EE_{k-1}))  \,c_{t_k}(\mathcal{Q}_k)\cdots c_{t_m}(\mathcal{Q}_m) \,\smallfrown\,
	\left[E_k\right]t_k\Big).
	\end{align*}
	Since $\mathcal{Q}_i$  is a quotient of the globally generated vector bundle $q^*(\EE_i)$, it is also globally generated.
	Recall that $c_\ttt(\underline{\mathcal{E}}) = \prod_{i,j} \left(1+d_{i,j}\,H\,t_i\right)$. 
	Therefore the nonnegativity claim follows  from \cite[Example 12.1.7(a)]{FULTON_INTERSECTION_THEORY}.
\end{proof}

Finally, we are ready to prove \autoref{thm_rationality}.

\begin{proof}[Proof of \autoref{thm_rationality}]
	Due to \autoref{lem_existence}, we may assume $n \ge m+ \sum_{i=1}^{m}r_i$.
	Thus \autoref{lem_truncation} yields 
	$$
	{\rm T}_N\left(\zeta_{I_1, \ldots, I_m}(\ttt)\right) \;=\; {\rm T}_N\left(
	\frac{
		{\rm T}_n\Big(
		\prod_{i,j} \left(1+d_{i,j}\,t_i\right)	\zeta_{I_1, \ldots, I_m}(\ttt)
		\Big)
	}{\prod_{i,j} \left(1+d_{i,j}\,t_i\right)}
	\right)
	$$
	for all $N \ge n$.
	Since $P(\ttt):={\rm T}_n\big(
	\prod_{i,j} \left(1+d_{i,j}\,t_i\right)	\zeta_{I_1, \ldots, I_m}(\ttt)
	\big)$ is a polynomial not depending on $N$, we obtain the equality
	$$
	\zeta_{I_1, \ldots, I_m}(\ttt) \;=\; \frac{P(\ttt)}{\prod_{i,j} \left(1+d_{i,j}\,t_i\right)}.
	$$
	This shows the rationality of the mixed Segre zeta function $\zeta_{I_1, \ldots, I_m}(\ttt)$.
	The fact that $P(\ttt)$ has nonnegative coefficients follows from \autoref{lem_nonnegative}.
\end{proof}

The next proposition yields a mixed formula relating the mixed Segre zeta function of the ideals $I_1,\ldots,I_m$ with the Segre zeta function of their product. 

\begin{proposition}[Mixed formula]
	\label{prop_mixed_formula}
	Consider the product ideal $I = I_1\cdots I_m \subset R$.
	Write 
	$$
	\zeta_{I_1, \ldots, I_m}(t_1,\ldots,t_m) \;=\; \sum_{i_1,\ldots,i_m \in \NN} \, a_{i_1,\ldots,i_m} \, t_1^{i_1}\cdots t_m^{i_m}
	$$
	and define the power series
	$$
	\widetilde{\zeta}_{I_1, \ldots, I_m}(t_1,\ldots,t_m) \;:=\; \sum_{i_1,\ldots,i_m \in \NN} \, \frac{(i_1+\cdots+i_m)!}{i_1!\cdots i_m!}\, a_{i_1,\ldots,i_m} \, t_1^{i_1}\cdots t_m^{i_m} \;\in\; \ZZ[\![t_1,\ldots,t_m]\!].
	$$
	Then we obtain the equality 
	$$
	\zeta_I(t) \;=\; \widetilde{\zeta}_{I_1, \ldots, I_m}(t,\ldots,t).
	$$
\end{proposition}
\begin{proof}
	The result follows from \autoref{cor_mixed_formula}.
\end{proof}

Another interesting result about mixed Segre zeta functions is that they only depend on the integral closure of the ideals $I_1,\ldots,I_m$.

\begin{proposition}[Invariance under integral dependence]
	\label{prop_integral_dep}
	We have the equality 
	$$
	\zeta_{I_1, \ldots, I_m}(t_1,\ldots,t_m) \;=\;  \zeta_{\overline{I_1}, \ldots, \overline{I_m}}(t_1,\ldots,t_m),
	$$
	where $\overline{I_i} \subset R$ denotes the integral closure of $I_i$.
\end{proposition}
\begin{proof}
	Consider the closed subschemes $\overline{Z_i}  =V(\overline{I_i}) \subset \PP_\kk^n$ and $\overline{Z}  =V(\overline{I_1}\cdots\overline{I_m}) \subset \PP_\kk^n$ and the joint blow-up $\overline{b} : \overline{\mathcal{B}} = \Bl_{\overline{Z_1},\ldots,\overline{Z_m}}(\PP_\kk^n) \rightarrow \PP_\kk^n$.
	Let $\overline{E_i}$ be the pull-back to $\overline{\mathcal{B}}$ of the exceptional divisor of the blow-up $\Bl_{\overline{Z_i}}(\PP_\kk^n)$.
	Let $E = E_1+\cdots+E_m$ and $\overline{E} = \overline{E_1}+\cdots+\overline{E_m}$, and consider the natural morphisms $\psi : \overline{E} \rightarrow E$, $\eta : E \rightarrow Z$ and $\overline{\eta} : \overline{E}  \rightarrow \overline{Z}$.
	Under the  finite birational morphism $\overline{\mathcal{B}} \rightarrow \mathcal{B}$ the pull-back of $E_i$ is equal to $\overline{E_i}$; indeed, recall that $I_i{\overline{I_i}}^k = {\overline{I_i}}^{k+1}$ for $k \gg 0$. 
	Since $Z$ and $\overline{Z}$ have the same support, we may identify their Chow groups. 
	Then \autoref{thm_funct_mixed_Segre}(i) yields
	\begin{align*}
		s\left(Z_1,\ldots,Z_m;\,\PP_\kk^n\right) &\;=\; \eta_*\left(s\left(E_1,\ldots,E_m;\, \mathcal{B}\right)\right)\\
		&\;=\; \eta_*\left(\psi_*\left(s\left(\overline{E_1},\ldots,\overline{E_m};\, \overline{\mathcal{B}}\right)\right)\right) \\
		&\;=\; \overline{\eta}_*\left(s\left(\overline{E_1},\ldots,\overline{E_m};\, \overline{\mathcal{B}}\right)\right)\\
		&\;=\; s\left(\overline{Z_1},\ldots,\overline{Z_m};\,\PP_\kk^n\right) .
	\end{align*}
	The result of the proposition follows because $\overline{I_i} R^N = \overline{I_iR^N}$ for any $N \ge n$.
\end{proof}

\section{Log-concavity of mixed Segre zeta functions}
\label{sect_Lorentzian}

In this section,  we explore a surprising log-concavity behavior that mixed Segre zeta functions have.
Here we show that, for arbitrary homogeneous ideals $I_1,\ldots,I_m$,  the homogenization of the numerator of the modified mixed Segre zeta function $1-\zeta_{I_1,\ldots,I_m}(t_1,\ldots,t_m)$ is denormalized Lorentzian.

We first recall the notion of Lorentzian polynomials by Br\"and\'en and Huh \cite{BrandenHuh}.
Let $h(t_1,\dots,t_m)$ be a homogeneous polynomial of degree $d$ in $\RR[\ttt]=\RR[t_1,\ldots,t_m]$.

\begin{definition}[\cite{BrandenHuh}]
	The homogeneous polynomial $h$ is called \emph{Lorentzian} if the following conditions hold:
	\begin{enumerate}[\rm (i)]
		\item The coefficients of $h$ are nonnegative.
		\item The support of $h$ is M-convex.
		\item The quadratic form $\partial_{t_{i_1}}\partial_{t_{i_2}}\cdots \partial_{t_{i_e}}(h)$ has at most one positive eigenvalue for any $1 \le i_1 \le \cdots \le i_e \le m$ where $e = d-2$.
	\end{enumerate}		
\end{definition}
There are several linear operators discussed in \cite{BrandenHuh} that preserve the Lorentzian property. In particular, the normalization operator 
$$
N\left(\sum_{\bn} a_\bn\ttt^\bn\right) \; := \;\sum_{\bn} \frac{a_\bn}{\bn!}\ttt^\bn \qquad \text{ where }  \qquad \bn! \;:=\; n_1!\cdots n_m!,
$$ 
preserves the Lorentzian property \cite[Corollary 3.7]{BrandenHuh}.
A (not necessarily Lorentzian) polynomial $h \in \RR[\ttt] = \RR[t_1,\ldots,t_m]$ is called \emph{denormalized Lorentzian} if its normalization $N(h) \in \RR[\ttt]$ is a Lorentzian polynomial.

The next proposition is our main technical tool to show the Lorentzian property of a polynomial. 
This result yields a large class of Lorentzian polynomials by considering the total Chern classes of a sequence of globally generated vector bundles on an irreducible variety.
Similar results have appeared in \cite[Proposition 3.9]{aluffi2024lorentzian} and in \cite[\S 9]{BEST}.

\begin{proposition}
	\label{prop_Lorentzian}
	Let $\sX$ be an $N$-dimensional irreducible variety over an algebraically closed field $\kk$ and $q : \sX \rightarrow \PP_\kk^n$ be a proper morphism.
	Let $\mathcal{Q}_1,\ldots,\mathcal{Q}_m$ be globally generated vector bundles on $\sX$. 
	Consider the ring presentation $A^*(\PP_\kk^n)[t_1,\ldots,t_m] \cong \ZZ[H,t_1,\ldots,t_m]/\left(H^{n+1}\right)$ and write 
	$$
	q_*\Big(c_{t_1}(\mathcal{Q}_1) \cdots c_{t_m}(\mathcal{Q}_m) \,\smallfrown\, \left[\sX\right]\Big) \;=\; \sum_{i_1+\cdots+i_m\le N}  a_{i_1,\ldots,i_m} \, t_1^{i_1}\cdots t_m^{i_m}\, H^{n-N+(i_1+\cdots+i_m)} \,\smallfrown\, \left[\PP_\kk^n\right]
	$$
	{\rm(}as before, elements in $A^*(\PP_\kk^n)[t_1,\ldots,t_m]$ are seen as formal polynomial in the variables $t_1,\ldots,t_m$ and with coefficients in $A_*(\PP_\kk^n)${\rm)}.
	Then the polynomial 
	$$
	\sum_{i_1+\cdots+i_m\le N}  a_{i_1,\ldots,i_m} \, t_0^{N-(i_1+\cdots+i_m)} \, t_1^{i_1}\cdots t_m^{i_m} \;\,\in\;\,  \NN\left[t_0,t_1,\ldots,t_m\right]
	$$
	is denormalized Lorentzian.
\end{proposition}
\begin{proof}
	Let $r_i = \rank(\mathcal{Q}_i)$ be the rank of $\mathcal{Q}_i$. 
	Since $\mathcal{Q}_i$ is globally generated, we obtain a short exact sequence 
	$$
	0 \;\rightarrow\; \mathcal{K}_i \;\rightarrow\; \OO_{\sX}^{\ell_i+1} \;\rightarrow\; \mathcal{Q}_i \;\rightarrow\; 0
	$$
	of vector bundles on $\sX$, and we may assume that $\ell_i \ge r_i$.
	Consider the irreducible variety
	$$
	\mathscr{P} \;:=\; \PP\big(\OO_{\sX}^{\ell_1+1}\big) \times_{\sX} \cdots \times_{\sX} \PP\big(\OO_{\sX}^{\ell_m+1}\big) \;\cong\; \sX \times_\kk \PP_\kk^{\ell_1} \times_\kk \cdots \times_\kk \PP_\kk^{\ell_m}
	$$
	and the natural projections $p : \sP \rightarrow \sX$ and $p_i : \sP \rightarrow \PP_\kk^{\ell_i}$.
	Let $H_i$ be class of a hyperplane in $\PP_\kk^{\ell_i}$ and notice that $p_i^*H_i = c_1(\OO_{\sP}(\ee_i))$.
	The irreducible subvariety $\sK := \PP(\mathcal{K}_1) \times_{\sX} \cdots \times_{\sX} \PP(\mathcal{K}_m) \subset \sP$ is given as the zero-scheme of the regular section 
	$$
	\OO_\sP \;\rightarrow\; \big(p^*(\mathcal{Q}_1) \otimes \OO_\sP(\ee_1)\big) \,\oplus\, \cdots \,\oplus\, \big(p^*(\mathcal{Q}_m) \otimes \OO_\sP(\ee_m)\big);
	$$
	see \cite[B.5.6, Example 3.2.17]{FULTON_INTERSECTION_THEORY} or \cite[Proposition 9.13]{EH}.
	By \cite[Example 3.2.16]{FULTON_INTERSECTION_THEORY} and the Whitney sum formula, we obtain 
	\begin{align*}
		\left[\sK\right] &\;=\; c_{r_1+\cdots+r_m}\left(\big(p^*(\mathcal{Q}_1) \otimes \OO_\sP(\ee_1)\big) \,\oplus\, \cdots \,\oplus\, \big(p^*(\mathcal{Q}_m) \otimes \OO_\sP(\ee_m)\big)\right) \,\smallfrown\, \left[\sP\right] \\
		&\;=\; c_{r_1}\left(\big(p^*(\mathcal{Q}_1) \otimes \OO_\sP(\ee_1)\big)\right)  \cdots  c_{r_m}\left(\big(p^*(\mathcal{Q}_m) \otimes \OO_\sP(\ee_m)\big)\right) \,\smallfrown\, \left[\sP\right] \\
		&\;=\; \sum_{0 \le i_j \le r_j} \left(p_1^*H_1\right)^{r_1-i_1}\cdots \left(p_m^*H_m\right)^{r_m-i_m} \, c_{i_1}\left(p^*(\mathcal{Q}_1)\right) \cdots c_{i_m}\left(p^*(\mathcal{Q}_m)\right) \,\smallfrown\, \left[\sP\right].
	\end{align*}
	Consider the proper morphism 
	$$
	f = q \times {\rm id} \;:\; \sP = \sX  \times_\kk \PP_\kk^{\ell_1} \times_\kk \cdots \times_\kk \PP_\kk^{\ell_m} \;\rightarrow\; \PP_\kk^n \times_\kk \PP_\kk^{\ell_1} \times_\kk \cdots \times_\kk \PP_\kk^{\ell_m} =: \PP.
	$$
	Let $\pi : \PP \rightarrow \PP_\kk^n$ and $\pi_i : \PP \rightarrow \PP_\kk^{\ell_i}$ be the natural projections. 
	We have the following commutative diagram
		\begin{equation*}
		\begin{tikzpicture}[baseline=(current  bounding  box.center)]
			\matrix (m) [matrix of math nodes,row sep=4em,column sep=7em,minimum width=2em, text height=1.5ex, text depth=0.25ex]
			{
				\mathscr{P} \cong \sX \times_\kk \PP_\kk^{\ell_1} \times_\kk \cdots \times_\kk \PP_\kk^{\ell_m} & &  \PP = \PP_\kk^n \times_\kk \PP_\kk^{\ell_1} \times_\kk \cdots \times_\kk \PP_\kk^{\ell_m}\\
				& \PP_\kk^{\ell_i} & \\
				\sX & & \PP_\kk^n.\\
			};
			\path[-stealth]
			(m-1-1) edge node [above] {$f$} (m-1-3)
			(m-1-1) edge node [left] {$p$} (m-3-1)
			(m-1-1) edge node [above] {$p_i$} (m-2-2)
			(m-3-1) edge node [above] {$q$} (m-3-3)
			(m-1-3) edge node [right] {$\pi$} (m-3-3)
			(m-1-3) edge node [above] {$\pi_i$} (m-2-2)
			;
		\end{tikzpicture}	
	\end{equation*}
	The projection and pull-back formulas and \cite[Proposition 1.7]{FULTON_INTERSECTION_THEORY} yield
		\begin{align*}
		f_*\left(\left[\sK\right]\right) 
		&\;=\; f_*\left(\sum_{0 \le i_j \le r_j} \left(p_1^*H_1\right)^{r_1-i_1}\cdots \left(p_m^*H_m\right)^{r_m-i_m} \, c_{i_1}\left(p^*(\mathcal{Q}_1)\right) \cdots c_{i_m}\left(p^*(\mathcal{Q}_m)\right) \,\smallfrown\, \left[\sP\right]\right) \\
		&\;=\; f_*\left(\sum_{0 \le i_j \le r_j} \left(f^*\pi_1^*H_1\right)^{r_1-i_1}\cdots \left(f^*\pi_m^*H_m\right)^{r_m-i_m} \, p^*\Big(c_{i_1}\left(\mathcal{Q}_1\right) \cdots c_{i_m}\left(\mathcal{Q}_m\right) \,\smallfrown\, \left[\sX\right]\Big)\right) \\
		&\;=\; \sum_{0 \le i_j \le r_j} \left(\pi_1^*H_1\right)^{r_1-i_1}\cdots \left(\pi_m^*H_m\right)^{r_m-i_m} \, \pi^*\left(q_*\Big(c_{i_1}\left(\mathcal{Q}_1\right) \cdots c_{i_m}\left(\mathcal{Q}_m\right) \,\smallfrown\, \left[\sX\right]\Big)\right).
	\end{align*}
	Hence by assumption we obtain 
	$$
	f_*\left(\left[\sK\right]\right)  \;=\; \sum_{0 \le i_j \le r_j} a_{i_1,\ldots,i_m}\left(\pi_1^*H_1\right)^{r_1-i_1}\cdots \left(\pi_m^*H_m\right)^{r_m-i_m} \, \left(\pi^*H\right)^{n-N+(i_1+\cdots+i_m)} \,\smallfrown\, \left[\PP\right].
	$$
	Let $h = \pi^*H$, $h_i = \pi_i^*H_i$, and $d = \dim(\sK) = N + (\ell_1+\cdots+\ell_m)-(r_1+\cdots+r_m)$.
	Consider the irreducible subvariety $f(\sK) \subset \PP$ and write $f_*\left(\left[\sK\right]\right) = e \left[f\left(\sK\right)\right]$ for some nonnegative integer $e \ge 0$.
 	By \cite[Theorem 4.16]{BrandenHuh}, the volume polynomial 
	$$
	V(\ttt) \;:=\; \int \left(ht_0+h_1t_1+\cdots+h_mt_m\right)^d \,\smallfrown\, \left[f(\sK)\right] \;\;\in\;\; \NN[t_0,t_1,\ldots,t_m]
	$$
	is Lorentzian.
	We now expand to obtain the following equalities
	\begin{align*}
	eV(\ttt) &\;=\;	\int \left(ht_0+h_1t_1+\cdots+h_mt_m\right)^d \,\smallfrown\, f_*\left(\left[\sK\right]\right) \\
	&\;=\;	\int \sum_{i_0+i_1+\cdots+i_m=d}\frac{d!}{i_0!i_1!\cdots i_m!}\, h^{i_0}h_1^{i_1}\cdots h_m^{i_m}\,t_0^{i_0}t_1^{i_1}\cdots t_m^{i_m} \,\smallfrown\, f_*\left(\left[\sK\right]\right) \\
	&\;=\; d!\sum_{i_0+i_1+\cdots+i_m=d}\frac{1}{i_0!i_1!\cdots i_m!}\, a_{i_1+r_1-\ell_1,\,\ldots,\,i_m+r_m-\ell_m}  \,\,t_0^{i_0}t_1^{i_1}\cdots t_m^{i_m}.
	\end{align*}
	As taking derivatives preserves Lorentzian polynomials, it follows that the following polynomial 
	$$
	\frac{e}{d!} \partial_{t_1}^{\ell_1-r_1}\cdots\partial_{t_m}^{\ell_m-r_m}\left(V(\ttt)\right) \;=\; \sum_{i_0+i_1+\cdots+i_m=N}\frac{1}{i_0!i_1!\cdots i_m!}\, a_{i_1,\,\ldots,\,i_m}  \,\,t_0^{i_0}t_1^{i_1}\cdots t_m^{i_m}
	$$ 
	is also Lorentzian. 
	This concludes the proof of the proposition.
\end{proof}

The following theorem is the main result of this section. 
It shows that the homogenization of the numerator of $1 - \zeta_{I_1, \ldots, I_m}(t_1,\ldots,t_m)$ is denormalized Lorentzian.

\begin{theorem}
	\label{thm_Lorentzian}
	Let $R = \kk[x_0,\ldots,x_n]$ be a polynomial ring over an algebraically closed field $\kk$. 
	Let $I_1,\ldots,I_m \subset R$ be homogeneous ideals in $R$.
	Let $f_{i, 0}, f_{i, 1}, \dotsc, f_{i, r_i} \in R$ be homogeneous generators of $I_i$ with $d_{i,j} = \deg\left(f_{i,j}\right)$.
	Define the polynomial $Q(t_1,\ldots,t_m)$ by the identity
	$$
	1 - \zeta_{I_1, \ldots, I_m}(t_1,\ldots,t_m) \;=\; \frac{Q(t_1,\ldots,t_m)}{\;\prod_{1 \le i \le m} \left(1+d_{i,0}t_i\right)\left(1+d_{i,1}t_i\right)\cdots\left(1+d_{i,r_i}t_i\right)\;}.
	$$
	Then the homogenization of $Q(t_1,\ldots,t_m)$ is a denormalized Lorentzian polynomial.
\end{theorem}
\begin{proof}
	We may assume $n \ge m+ \sum_{i=1}^{m}r_i$.
	By \autoref{lem_truncation}, we obtain
	\begin{align*}
	Q(t_1,\ldots,t_m) &\;=\;  \prod_{i,j}(1+d_{i,j}t_i) - P(t_1,\ldots,t_m)\\ 
	&\;=\;  T_n\left(\prod_{i,j}(1+d_{i,j}t_i)\,\left(1-\zeta_{I_1, \ldots, I_m}(t_1,\ldots,t_m)\right)\right)
	\end{align*}
	where $P(\ttt)={\rm T}_n\big(
	\prod_{i,j} \left(1+d_{i,j}\,t_i\right)	\zeta_{I_1, \ldots, I_m}(\ttt)
	\big)$ is the numerator of $\zeta_{I_1, \ldots, I_m}(t_1,\ldots,t_m)$ in \autoref{thm_rationality}.
	Write $Q(t_1,\ldots,t_m)=\sum_{i_1+\cdots+i_m\le n}a_{i_1,\ldots,i_m}t_1^{i_1}\cdots t_m^{i_m}$.
	Then notice that we get
	$$
	c_\ttt(\underline{\mathcal{E}}) \,\smallfrown\, \big(\left[\PP_\kk^n\right]-{\iota}_*\left(s\left(Z_1,\ldots,Z_m;\, \PP_\kk^n\right)\right)\big) \;=\; \sum_{i_1+\cdots+i_m\le n}a_{i_1,\ldots,i_m} \,t_1^{i_1}\cdots t_m^{i_m} \, H^{i_1+\cdots+i_m} \,\smallfrown\, \left[\PP_\kk^n\right];
	$$
	here we continue using the notation introduced in \autoref{sect_mixed_Seg_zeta}.
	Due to \autoref{thm_blow_up_form} and the projection and Whitney sum formulas, we obtain
	\begin{align*}
		c_\ttt(\underline{\mathcal{E}}) \,\smallfrown\, \big(\left[\PP_\kk^n\right]-{\iota}_*\left(s\left(Z_1,\ldots,Z_m;\, \PP_\kk^n\right)\right)\big) &\;=\; b_*\left( \frac{c_{t_1}(b^*(\EE_1))\cdots c_{t_m}(b^*(\EE_m))}{c_{t_1}(\OO_{\mathcal{B}}(-\ee_1))\cdots c_{t_m}(\OO_{\mathcal{B}}(-\ee_m))} \,\smallfrown\, \left[\mathcal{B}\right] \right) \\
		&\;=\; b_*\Big( c_{t_1}(\mathcal{Q}_1)\cdots c_{t_m}(\mathcal{Q}_m) \,\smallfrown\, \left[\mathcal{B}\right] \Big)
	\end{align*}
	where $\mathcal{Q}_i$ is equal to the quotient bundle $b^*(\EE_i)/\OO_\mathcal{B}(-\ee_i)$ on $\mathcal{B}$.
	Finally, since each $\mathcal{Q}_i$ is globally generated, the result of the theorem follows from \autoref{prop_Lorentzian}.
\end{proof}

\section{Examples and computations}
\label{sect_examples}

In this final section, we provide some examples and computations. 
In the first example, we give a general formula for the mixed Segre zeta function $\zeta_{I_1,I_2}(t_1,t_2)$ of two complete intersections $I_1$ and $I_2$ that involve disjoint sets of variables. 
In this case, the formula for the mixed Segre zeta function $\zeta_{I_1,I_2}(t_1,t_2)$ can be nicely expressed in terms of the Segre zeta functions of $I_1$ and $I_2$.

\begin{example}
	\label{examp_disjoint}
	Let $R = \kk[x_0,\ldots,x_n]$ be a polynomial ring over a field $\kk$ and $\PP_\kk^n = \Proj(R)$.
	Let $I_1=(f_1,\ldots,f_r)$ and $I_2 = (g_1,\ldots,g_s)$ be two homogeneous ideals of $R$.
	Let $a_i = \deg(f_i)$ and $b_i = \deg(g_i)$.
	Assume the following two conditions: 
	\begin{enumerate}[\rm (i)]
		\item $I_1$ and $I_2$ are complete intersections (i.e., $f_1,\ldots,f_r$ and $g_1,\ldots,g_s$ are regular sequences).
		\item There is an index $k$ such that $f_i \in \kk[x_0,\ldots,x_k]$ and $g_i \in \kk[x_{k+1},\ldots,x_n]$ (i.e., the generators of $I_1$ and $I_2$ involve disjoint sets of variables).
	\end{enumerate}
	Then the mixed Segre zeta function of $I_1$ and $I_2$ is given by 
	$$
	\zeta_{I_1, I_2}(t_1,t_2) \;=\;	\zeta_{I_1}(t_1)
		\;+\;	\zeta_{I_2}(t_2)
		\;-\;	
		\zeta_{I_1}(t_1)\zeta_{I_2}(t_2) \;\in\; \ZZ[\![t_1,t_2]\!],
	$$	
	where 
	$$
	\zeta_{I_1}(t_1) \;=\; \frac{a_1\cdots a_rt_1^r
	}{\prod_{i=1}^{r}\left(1+a_it_1\right)} \;\in\; \ZZ[\![t_1]\!] \quad\text{ and }\quad \zeta_{I_2}(t_2) \;=\; \frac{b_1\cdots b_s  t_2^s}{\prod_{i=1}^{s}\left(1+b_it_2\right)} \;\in\; \ZZ[\![t_2]\!]
	$$
	are the Segre zeta functions of $I_1$ and $I_2$, respectively.
\end{example}
\begin{proof}
	Let $Z_1 = V(I_1) \subset \PP_\kk^n$ and $Z_2 = V(I_2) \subset \PP_\kk^n$ be the closed subschemes of $\PP_\kk^n$ determined by $I_1$ and $I_2$, respectively. 
	Consider the joint blow-up $b : \mathcal{B} = \Bl_{Z_1,Z_2}(\PP_\kk^n) \rightarrow \PP_\kk^n$.
	Since the generators of $I_1$ and $I_2$ involve disjoint sets of variables, we obtain the following fibre square 
	\begin{equation*}
		\begin{tikzpicture}[baseline=(current  bounding  box.center)]
			\matrix (m) [matrix of math nodes,row sep=4em,column sep=7em,minimum width=2em, text height=1.5ex, text depth=0.25ex]
			{
				\mathcal{B}=\Bl_{Z_1,Z_2}(\PP_\kk^n) &  \mathcal{B}_2=\Bl_{Z_2}(\PP_\kk^n)\\
				\mathcal{B}_1=\Bl_{Z_1}(\PP_\kk^n) & \PP_\kk^n\\
			};
			\path[-stealth]
			(m-1-1) edge node [left] {$\pi_1$} (m-2-1)
			(m-1-2) edge node [right] {$b_2$} (m-2-2)
			(m-1-1) edge node [above] {$\pi_2$}(m-1-2)
			(m-2-1) edge node [above] {$b_1$}(m-2-2)
			;
		\end{tikzpicture}	
	\end{equation*}
	(in general, we only have a commutative diagram); see, e.g., \cite[Corollary 2.4]{AF}.
	Let $E_i$ be the exceptional divisor of the blow-up $\mathcal{B}_i=\Bl_{Z_i}(\PP_\kk^n)$.
	By an abuse of notation, the pull-back of $E_i$ to $\mathcal{B} = \Bl_{Z_1,Z_2}(\PP_\kk^n)$ is also denoted as $E_i$.
	Let $W = V(I_1+I_2) \subset \PP_\kk^n$ be the closed subscheme of $\PP_\kk^n$ determined by $I_1+I_2$ and consider the closed immersion $j_i : W \hookrightarrow Z_i$.
	Since $Z_i$ is a complete intersection,  we obtain the isomorphisms 
	$$
	b_i^{-1}(W) \;=\; \mathcal{B}_i \times_{\PP_\kk^n} W  \;\cong\; E_i \times_{Z_i} W \;\cong\; \PP\left(\mathcal{N}_{Z_i}\PP_\kk^n\right) \times_{Z_i} W \;\cong\;  \PP\left(j_i^*\left(\mathcal{N}_{Z_i}\PP_\kk^n\right)\right)
	$$
	where $\mathcal{N}_{Z_i}\PP_\kk^n$ is the normal bundle to $Z_i$ in $\PP_\kk^n$.
	Hence restricting the above fibre square to $W$ yields the new fibre square 
		\begin{equation*}
		\begin{tikzpicture}[baseline=(current  bounding  box.center)]
			\matrix (m) [matrix of math nodes,row sep=4em,column sep=7em,minimum width=2em, text height=1.5ex, text depth=0.25ex]
			{
				b^{-1}(W) &  b_2^{-1}(W) \;\cong\; \PP\left(j_2^*\left(\mathcal{N}_{Z_2}\PP_\kk^n\right)\right)\\
				b_1^{-1}(W) \;\cong\; \PP\left(j_1^*\left(\mathcal{N}_{Z_1}\PP_\kk^n\right)\right) & W.\\
			};
			\path[-stealth]
			(m-1-1) edge (m-2-1)
			(m-1-2) edge (m-2-2)
			(m-1-1) edge (m-1-2)
			(m-2-1) edge (m-2-2)
			;
		\end{tikzpicture}	
	\end{equation*}
	Consider the closed immersions $j_{Z_i} : Z_i \hookrightarrow \PP_\kk^n$ and $j_W  : W \hookrightarrow \PP_\kk^n$.
	Since $[W] = a_1\cdots a_r b_1 \cdots b_s H^{r+s}$, $j_1^*\left(\mathcal{N}_{Z_1}\PP_\kk^n\right) = \bigoplus_{i=1}^r\OO_W\left(a_i\,j_W^*H\right)$ and $j_2^*\left(\mathcal{N}_{Z_2}\PP_\kk^n\right) = \bigoplus_{i=1}^s\OO_W\left(b_i\,j_W^*H\right)$, we obtain that 
	\begin{align*}
		b_*\left(\sum_{i,j\ge 1}{(-E_1)}^{i}t_1^i{(-E_2)}^{j}t_2^j \, \smallfrown\, \left[\mathcal{B}\right]\right) &\;=\; t_1t_2\,b_*\left(\sum_{i,j\ge 0}{(-E_1)}^{i}t_1^i{(-E_2)}^{j}t_2^j \, \smallfrown\, \left[b^{-1}(W)\right]\right)\\
		&\;=\; {j_{W}}_*\Big(t_1t_2 \, t_1^{r-1}t_2^{s-1} \, s_{t_1}\!\left(j_1^*\left(\mathcal{N}_{Z_1}\PP_\kk^n\right)\right) \, s_{t_2}\!\left(j_2^*\left(\mathcal{N}_{Z_2}\PP_\kk^n\right)\right) \,\smallfrown\, \left[W\right]\Big) \\
		&\;=\; 	\frac{a_1\cdots a_r\, b_1 \cdots b_s H^{r+s}  t_1^{r}t_2^{s} }{\prod_{i=1}^{r}(1+a_iHt_1)\,\prod_{i=1}^{s}(1+b_iHt_2)};
	\end{align*}
	see \cite[Proposition 4.1(a)]{FULTON_INTERSECTION_THEORY}.
	Similarly, we have that 
	$$
	-b_*\left(\sum_{i\ge 1}{(-E_1)}^{i}t_1^i \, \smallfrown\, \left[\mathcal{B}\right]\right) \;=\; {j_{Z_1}}_* \big( t_1 \, t_1^{r-1}\, s_{t_1}\!\left(\mathcal{N}_{Z_1}\PP_\kk^n\right) \,\smallfrown\, \left[Z_1\right]\big) \;=\; \frac{a_1\cdots a_r H^{r}  t_1^{r} }{\prod_{i=1}^{r}(1+a_iHt_1)}
	$$
	and 
	$$
	-b_*\left(\sum_{i\ge 1}{(-E_2)}^{i}t_2^i \, \smallfrown\, \left[\mathcal{B}\right]\right) \;=\; {j_{Z_2}}_*\big(t_2 \, t_2^{s-1}\, s_{t_2}\!\left(\mathcal{N}_{Z_2}\PP_\kk^n\right) \,\smallfrown\, \left[Z_2\right]\big)  \;=\; \frac{b_1\cdots b_s H^{s}  t_2^{s} }{\prod_{i=1}^{s}(1+b_iHt_2)}.
	$$
	By combining the above computations with \autoref{rem_seg_fulton}, we get 
	\begin{align*}
		\iota_*\left(s(Z_1,Z_2; \PP_\kk^n)\right) &\;=\; -b_*\left(\sum_{i+j\ge 1}(-E_1)^it_1^i(-E_2)t_2^j \,\smallfrown\, [\mathcal{B}]\right)  \\
		&\;=\;	
		\frac{a_1\cdots a_r H^r t_1^r
		}{\prod_{i=1}^{r}\left(1+a_iHt_1\right)}
		+
		\frac{b_1\cdots b_s H^s t_2^s}{\prod_{i=1}^{s}\left(1+b_iHt_2\right)}
		-	
		\frac{a_1\cdots a_r b_1\cdots b_s H^{r+s} t_1^r t_2^s}{\prod_{i=1}^{r}\left(1+a_iHt_1\right)\prod_{i=1}^{s}\left(1+b_iHt_2\right)}.
	\end{align*}
	The result of the example now follows (see \autoref{lem_existence}).
\end{proof}

By utilizing \autoref{examp_disjoint} and \autoref{prop_mixed_formula}, we obtain a formula for the Segre zeta function $\zeta_I(t)$ of the product $I = I_1I_2 \subset R$ of the two ideals $I_1$ and $I_2$.
We explore this observation in the next three examples.

\begin{example}
	\label{examp_binom_conv}
	We continue with the same setting of \autoref{examp_disjoint}. 
	Consider the ideal $I = I_1I_2 \subset R$ given as the product of $I_1$ and $I_2$.
	Write $A(t) = \zeta_{I_1}(t) - 1 = \sum_{i\ge 0}\alpha_i t^i \in \ZZ[\![t]\!]$ and $B(t) = \zeta_{I_2}(t) - 1 = \sum_{i\ge 0}\beta_i t^i \in \ZZ[\![t]\!]$.
	By combining \autoref{prop_mixed_formula} and \autoref{examp_disjoint}, we obtain the following formula 
	$$
	\zeta_I(t) \;=\; 1 \,-\, A(t) \odot B(t)
	$$
	 for the Segre zeta function of $I$ in terms of the \emph{binomial convolution} 
	 $$
	 A(t) \odot B(t) \;:=\; \sum_{k=0}^{\infty}\left(\sum_{i=0}^{k}\binom{k}{i}\alpha_i\beta_{k-i}\right)t^k 
	 $$ 
	 of $A(t)$ and $B(t)$.
	 By \cite[Theorem 6]{GK}, we can write 
	 $$
	 A(t) \odot B(t) \;=\; \frac{L(t)}{\prod_{i=1}^r\prod_{j=1}^s(1+(a_i+b_j)t)}
	 $$
	 for some polynomial $L(t)$.
	 Therefore we can also write 
	 $$
	 \zeta_I(t) \;=\; \frac{P(t)}{\prod_{i=1}^r\prod_{j=1}^s(1+(a_i+b_j)t)}
	 $$
	 for some polynomial $P(t)$.
	 This agrees with the general presentation of \autoref{thm_rationality} because $I$ is generated by the homogeneous polynomials $f_ig_j$ of degree $\deg(f_ig_j) = a_i + b_j$.
	 Consider the power series 
	 $$
	 E_A(t) \;=\; \sum_{i\ge 0}\frac{\alpha_i}{i!} t^i \quad
	  \text{ and } 
	  \quad
	  E_B(t) \;=\;  \sum_{i\ge 0}\frac{\beta_i}{i!} t^i.
	  $$
	 Then the binomial convolution can be computed as the following integral 
	 $$
	 A(t) \odot B(t) \;=\; \int_{0}^{\infty} E_A(tx)E_B(tx) \,e^{-x} \,dx.
	 $$
	 We shall exploit this representation as an integral in the next, more explicit example.
\end{example}

\begin{example}
	\label{exmp_r_1_s_2}
	With the same setting of \autoref{examp_disjoint}, we assume that $I_1= (f_1)$, $I_2=(g_1,g_2)$ and $b_1 \neq b_2$.
	Consider the following power series with their partial fraction decompositions 
	$$
	A(t) \;=\; \zeta_{I_1}(t) - 1 \;=\; -\frac{1}{1+a_1t}
	$$
	and 
	$$
	B(t) \;=\; \zeta_{I_2}(t) -1 \;=\; \frac{b_2/(b_1 - b_2)}{(1+b_1 t)} + \frac{b_1/(b_2 - b_1)}{(1+b_2 t)}.
	$$
	Let $\delta_1 = b_1/(b_2 - b_1)$ and $\delta_2 = b_2/(b_1 - b_2)$.
	Notice that 
	$$
	E_A(t) = -e^{-a_1t} \quad \text{ and } \quad
	E_B(t) = \delta_2 e^{-b_1t} + \delta_1e^{-b_2t}.
	$$
	We can now compute 
	\begin{align*}
		\zeta_I(t) &\;=\; 1 \,-\, A(t) \odot B(t)\\
		&\;=\;1 + \int_{0}^{\infty} e^{-a_1tx}\left(\delta_2 e^{-b_1tx} + \delta_1e^{-b_2tx}\right) e^{-x} \,dx\\
		&\;=\; 1+\int_{0}^{\infty} \delta_2e^{-\left(1+(a_1+b_1)t\right)x} + \delta_1e^{-\left(1+(a_1+b_2)t\right)x} \,dx\\
		&\;=\; 1 + \frac{\delta_2}{1+(a_1+b_1)t} + \frac{\delta_1}{1+(a_1+b_2)t},
	\end{align*}		
	which yields an explicit formula for $\zeta_I(t)$.
\end{example}

\begin{example}
	\label{examp_explicit}
	From \autoref{exmp_r_1_s_2}, in the case $I_1 = (x_0^2)$ and $I_2=(x_1^2, x_2^3)$, the Segre zeta function of $I = I_1I_2 = (x_0^2x_1^2, x_0^2x_2^3)$ is given by 
	$$
	\zeta_I(t) \;=\; 1 - \frac{3}{1+4t} + \frac{2}{1+5t} \;=\; \frac{2t+20t^2}{(1+4t)(1+5t)}.
	$$
	One can check that this agrees with the output of the \texttt{Macaulay2} package \cite{SegreClassesSource}.
	Let $I = (x_0^2x_1^2, x_0^2x_2^3) \subset \kk[x_0,\ldots,x_{10}]$ and $Z \subset \PP_\kk^{10}$ be the closed subscheme determined by $I$.
	By calling the function \texttt{segre} we obtain the output 
	$$
		\iota_*\left(s(Z, \PP_\kk^{10})\right) \;=\; \left(
		\begin{array}{c}
			16\,385\,522\,H^{10}-3\,119\,818\,H^{9}+584\,642\,H^{8}  
			-107\,098\,H^{7} \\ +18\,962\,H^{6}-3\,178\,H^{5}+482\,H^{4}-58\,H^{3}+2\,H^{2}+2\,H
		\end{array}
		\right) \;\smallfrown\; \left[\PP_\kk^{10}\right].
	$$
	The coefficients of $\iota_*\left(s(Z, \PP_\kk^{10})\right)$ give the first 10 terms of $\zeta_I(t)=\frac{2t+20t^2}{(1+4t)(1+5t)}$.
	The numerator $1 - \zeta_I(t)$ is equal to $Q(t) = 1 + 7t$.
	The homogenization $t_0 + 7t_1$ of $Q(t)$ is a linear form with nonnegative coefficients, and thus a Lorentzian polynomial.
\end{example}

Finally, we present some computations with the computer algebra system \texttt{Macaulay2} \cite{MACAULAY2}.

\begin{remark}
	\label{rem_algorithm}
	Similarly to the \texttt{Macaulay2} package  for Segre classes \cite{SegreClassesSource}, it would be useful to have a robust implementation to compute mixed Segre classes. 
	By utilizing \autoref{cor_computing}, we implemented a \texttt{Macaulay2}  function that computes mixed Segre classes.
	The implementation is publicly available at: 
	\begin{center}
		\url{https://ycid.github.io/mixed_Segre.m2}
	\end{center}
	In this resource, the function \texttt{mixedSegre} can be called to compute mixed Segre classes. 
\end{remark}

\begin{example}
	With the same setting of \autoref{examp_disjoint}, we assume that $R = \kk[x_0,\ldots,x_5]$, $I_1 = (x_0^2, x_1^2)$ and $I_2=(x_2^2, x_3^3)$.
	Let $Z_1$ and $Z_2$ be the closed subschemes of $\PP_\kk^5$ determined by $I_1$ and $I_2$, respectively.
	By calling the function \texttt{mixedSegre} described in \autoref{rem_algorithm}, we obtain the following output
	$$	
	\iota_*\left(s(Z_1,Z_2;\PP_\kk^5)\right)  \;=\;
	\left(
	\begin{array}{c}
		-128\,H^{5}t_{1}^{5}+96\,H^{5}t_{1}^{3}t_{2}^{2}+120\,H^{5}t_{1}^{2}t_{2}^{3}-390\,H^{5}t_{2}^{5}+48\,H^{4}t_{1}^{4}\\-24\,H^{4}t_{1}^{2}t_{2}^{2}+114\,H^{4}t_{2}^{4}-16\,H^{3}t_{1}^{3}-30\,H^{3}t_{2}^{3}+4\,H^{2}t_{1}^{2}+6\,H^{2}t_{2}^{2}
	\end{array}
	\right)
	\smallfrown
	\left[\PP_\kk^5\right].
	$$
	in $A^*\left(\PP_\kk^5\right)
	[t_1,t_2]\cong \ZZ[H,t_1,t_2]/\left(H^6\right)$.
	As predicted by \autoref{examp_disjoint}, this coincides with the terms of 
	\begin{align*}
	\zeta_{I_1, I_2}(t_1,t_2) &\;=\; \frac{4t_1^2}{(1+2t_1)^2} + \frac{6t_2^2}{(1+2t_2)(1+3t_2)} -\frac{24t_1^2t_2^2}{(1+2t_1)^2(1+2t_2)(1+3t_2)} \\
	&\;=\; \frac{4\,t_1^{2}+6\,t_2^{2}+20\,t_1^{2}t_2+24\,t_1t_2^{2}+24\,t_1^{2}t_2^{2}}{(1+2t_1)^2(1+2t_2)(1+3t_2)}
	\end{align*}
	of total degree at most $5$.
	The numerator of $1 - \zeta_{I_1,I_2}(t_1,t_2)$ is equal to $Q(t_1,t_2) = 1 +4t_1 +5t_2 +  20t_1t_2$.
	It can be checked that the homogenization $t_0^2 +4t_0t_1 +5t_0t_2 +  20t_1t_2$ of $Q(t_1,t_2)$ is a Lorentzian polynomial.
\end{example}

\section*{Acknowledgments}

We are grateful to Matt Larson for several helpful conversations. 
We thank the reviewer for carefully reading our paper and for several comments and corrections.
The author was partially supported by NSF grant DMS-2502321.

\bibliography{references}

\end{document}